\documentclass[10pt]{article}
\usepackage{amssymb}
\usepackage{graphicx}
\usepackage{xcolor} 
\usepackage{tensor}
\usepackage{fullpage} 
\usepackage{amsmath}
\usepackage{amsthm}
\usepackage{verbatim}
\usepackage{hyperref}
\usepackage{cite}
\usepackage{array} 

\usepackage{enumitem}
\setlist[enumerate]{leftmargin=1.2em} 
\setlist[itemize]{leftmargin=1.2em}
\usepackage{seqsplit,mathtools}

\definecolor{green}{rgb}{0,0.8,0} 

\newtheorem{theorem}{Theorem}[section]

\newtheorem{lemma}[theorem]{Lemma}
\newtheorem{conjecture}[theorem]{Conjecture}
\newtheorem{proposition}[theorem]{Proposition}

\theoremstyle{definition}

\newtheorem{example}[theorem]{Example}
\theoremstyle{remark}
\newtheorem{remark}[theorem]{Remark}

\numberwithin{equation}{section}
\newcommand{\nrm}[1]{\Vert#1\Vert}

\newcommand{\nnrm}[1]{{\vert\kern-0.25ex\vert\kern-0.25ex\vert #1 
		\vert\kern-0.25ex\vert\kern-0.25ex\vert}}

\newcommand{\sgn}{{\mathrm{sgn}}\,}
\newcommand{\supp}{{\mathrm{supp}}\,}

\renewcommand{\Re}{\mathrm{Re}}
\renewcommand{\Im}{\mathrm{Im}}

\newcommand{\lap}{\Delta}

\newcommand{\rd}{\partial}
\newcommand{\nb}{\nabla}

\newcommand{\alp}{\alpha}
\newcommand{\bt}{\beta}

\newcommand{\Gmm}{\Gamma}
\newcommand{\dlt}{\delta}

\newcommand{\eps}{\epsilon}

\newcommand{\lmb}{\lambda}

\newcommand{\tht}{\theta}

\newcommand{\omg}{\omega}

\newcommand{\varep}{\varepsilon}

\newcommand{\bfI}{{\bf I}}

\newcommand{\bfK}{{\bf K}}

\newcommand{\bbK}{\mathbb K}

\newcommand{\bbR}{\mathbb R}
\newcommand{\bbS}{\mathbb S}

\newcommand{\bbZ}{\mathbb Z}

\newcommand{\calD}{\mathcal D}

\newcommand{\calM}{\mathcal M}

\begin{document}
	
	\bibliographystyle{plain}
	\title{Logarithmic spirals in 2d perfect fluids}
	\author{In-Jee Jeong\thanks{Department of Mathematical Sciences and RIM, Seoul National University.   E-mail: injee\_j@snu.ac.kr}  \and Ayman R. Said\thanks{Department of Mathematics, Duke University.   E-mail: ayman.said@duke.edu} }
	
	\date\today
	
	\maketitle
	
	\renewcommand{\thefootnote}{\fnsymbol{footnote}}
	\renewcommand{\thefootnote}{\arabic{footnote}}

	\begin{abstract} 
		We study logarithmic spiraling solutions to the 2d incompressible Euler equations which solve a nonlinear transport system on $\mathbb{S}$. We show that this system is locally well-posed in $L^p, p\geq 1$ as well as for atomic measures, that is logarithmic spiral vortex sheets. Logarithmic spiral vortex sheets introduced by Prandtl and Alexander are just particular examples in this well-posedness class which satisfy an additional scaling symmetry involving time. Moreover, we realize the dynamics of logarithmic vortex sheets as the well-defined limit of solutions which are smooth in the angle. Furthermore, our formulation not only allows for a simple proof of existence and bifurcation for non-symmetric multi-branched logarithmic spiral vortex sheets, but also provides a framework for studying asymptotic stability of self-similar dynamics. 

        For logarithmic spiraling solutions, we make an observation that the local circulation of the vorticity around the origin is a strictly monotone quantity of time, which allows for a rather complete characterization of the long-time behavior. We prove global well-posedness for bounded logarithmic spirals as well as data that admit at most logarithmic singularities. We are then able to show a dichotomy in the long time behavior, solutions either blow up (either in finite or infinite time) or completely homogenize. In particular, bounded logarithmic spirals should converge to constant steady states. For logarithmic spiral sheets, the dichotomy is shown to be even more drastic, where only finite time blow up or complete homogenization of the fluid can and does occur.
	\end{abstract}
	
	\section{Introduction} 
	
	\subsection{Logarithmic spirals}
	The vorticity equation for incompressible and inviscid fluids in $\bbR^2$ is given by  \begin{equation}\label{eq:2DEuler-vort}
		\left\{
		\begin{aligned}
			\rd_t\omg + u\cdot\nb \omg = 0, & \\
			u = \nb^\perp \lap^{-1}\omg , &
		\end{aligned}
		\right.
	\end{equation} where $\omg(t,\cdot):\bbR^2\to\bbR$ and $u(t,\cdot):\bbR^2\to\bbR^{2}$ denote the vorticity and velocity of the fluid, respectively. In this paper, we are concerned with solutions of \eqref{eq:2DEuler-vort} which are supported on \textit{logarithmic spirals}; in other words, vorticities $\omg$ which are \textit{invariant} under the following group of transformations of $\bbR^{2}$ parametrized by $\lmb>0$ \begin{equation*}
		\begin{split}
			(r,\tht) \mapsto (\lmb r, \tht + \beta\ln \lmb),
		\end{split}
	\end{equation*} for some nonzero real constant $\bt$. Here, $(r,\tht)$ denotes the polar coordinates in $\bbR^2$. Under the above invariance, we can take a periodic function $h(t,\cdot)$ of one variable such that 
	\begin{equation}\label{eq:vort-homog-spi}
		\begin{split}
			\omg(t,r,\tht) = h(t, \tht - \beta \ln r )
		\end{split}
	\end{equation} holds for all $t, r, \tht$. Then, the two-dimensional PDE \eqref{eq:2DEuler-vort} reduces to the  one-dimensional transport equation in terms of $h$: \begin{equation}\label{eq:vort-evo}
		\begin{split}
			\rd_t h + 2H \rd_\tht h = 0 
		\end{split}
	\end{equation} coupled with the elliptic problem \begin{equation}\label{eq:elliptic}
		\begin{split}
			4H-4\beta \partial_{\theta} H + (1+\beta^2)\partial^2_{\theta} H = h
		\end{split}
	\end{equation} defined on $\bbS = \bbR/(2\pi\bbZ)$. To derive \eqref{eq:vort-evo}--\eqref{eq:elliptic}, we take the following ansatz for the stream function 
	\begin{equation}\label{eq:stream-homog}
		\begin{split}
			\Psi(t,r,\tht) = r^2 H(t, \tht - \beta \ln r ),
		\end{split}
	\end{equation} and note that the corresponding velocity field $u = u^re^r+u^\tht e^\tht$ is given by \begin{equation}\label{eq:vel-homog}
		\left\{
		\begin{aligned}
			u^r(t,r,\tht) & = -r^{-1}\rd_\tht\Psi = - r\partial_{\theta} H(t, \tht - \beta \ln r ), \\
			u^\tht(t,r,\tht) & = \rd_r\Psi = 2rH(t, \tht - \beta \ln r ) - \bt r \partial_{\theta} H(t, \tht - \beta \ln r ),
		\end{aligned}
		\right.
	\end{equation} which gives that $u\cdot\nb\omg = 2H\rd_\tht h$. Furthermore, \eqref{eq:elliptic} follows from $\omg = \nb\times u = \frac1r ( \rd_r(ru^\tht) - \rd_\tht u^r )$ and \eqref{eq:vel-homog}. As we shall explain below, the special case $\bt=0$ corresponds to the system for 0-homogeneous vorticity studied in \cite{EJ1,EMS}. In this paper, we shall study dynamical properties of the system \eqref{eq:vort-evo}--\eqref{eq:elliptic} for $\bt\ne0$. 
	
	\subsection{Main results}
	
	Let us present the main results of this paper, which come in two categories: well-posedness issues and long-time dynamics. 
	
	\medskip
	
	\noindent \textbf{Well-posedness of logarithmic vortex.} To begin with, we have local and global existence and uniqueness of the system \eqref{eq:vort-evo}--\eqref{eq:elliptic} for $h\in L^p$ and $L^\infty$, respectively. 
	\begin{theorem}[Local well-posedness]\label{thm:wp-Lp}
		We have the following well-posedness results for the initial value problem of \eqref{eq:vort-evo}--\eqref{eq:elliptic}: if $h_0 \in L^{p}(\bbS)$ for some $1\le p \le \infty$, there exist $T\ge c\nrm{h_0}_{L^{1}}^{-1}$ and a unique local solution $h$ belonging to $C([0,T);L^p(\bbS))$ for $p<\infty$ and $C_*([0,T);L^\infty(\bbS))$ for $p=\infty$, $c>0$ is a universal constant. For any $p$, the $L^p$ solution blows up at $T^*$ if and only if \begin{equation*}
			\begin{split}
				\int_0^{T^*} \nrm{h(t,\cdot)}_{L^{1}} dt = \infty. 
			\end{split}
		\end{equation*} 
	\end{theorem}

    \begin{remark}
        Higher regularity (e.g. $C^{k,\alp}, W^{k,p}$ for $k\ge0$) of $h$ propagates in time as long as $\nrm{h(t,\cdot)}_{L^{1}}$ remains bounded. 
    \end{remark}
	
	In the case of bounded, or even logarithmically singular data, we have global regularity: \begin{theorem}[Global well-posedness]\label{thm:gwp}
		The local solution in Theorem \ref{thm:wp-Lp} is global in time for initial data $h_0 \in \cap_{p<\infty} L^p$ satisfying \begin{equation}\label{eq:Lp-log}
			\begin{split}
				\sup_{p\ge 1} \frac{\nrm{h_0}_{L^p}}{p} < + \infty . 
			\end{split}
		\end{equation} In particular, $L^\infty$ solutions are global in time. 
	\end{theorem}
	
	\begin{remark} The above can be extended to data satisfying $\nrm{h_0}_{L^p} \lesssim p\ln (10+p)$, $\lesssim p \ln (10+p) \ln(10+ \ln(10+p))$, and so on. \end{remark}

	\begin{figure}
		\centering
		\includegraphics[scale=0.4]{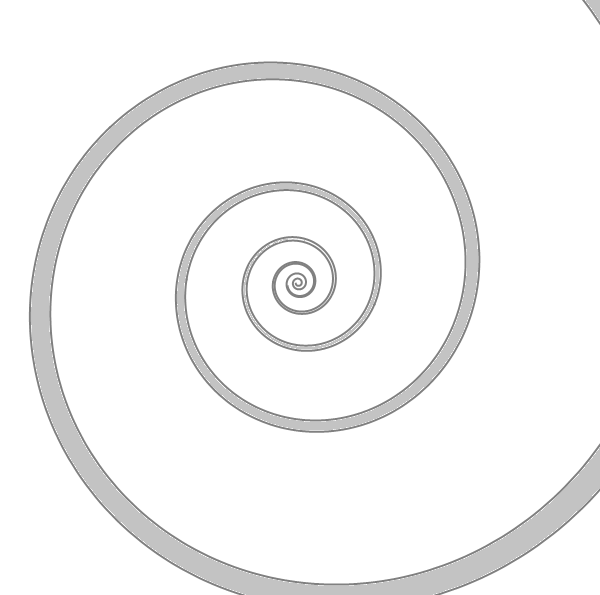}  \,\, 		\includegraphics[scale=0.4]{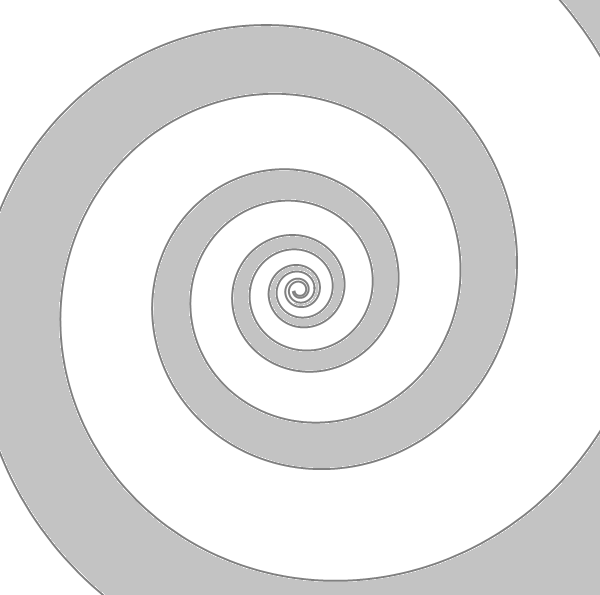}  \,\, 		\includegraphics[scale=0.4]{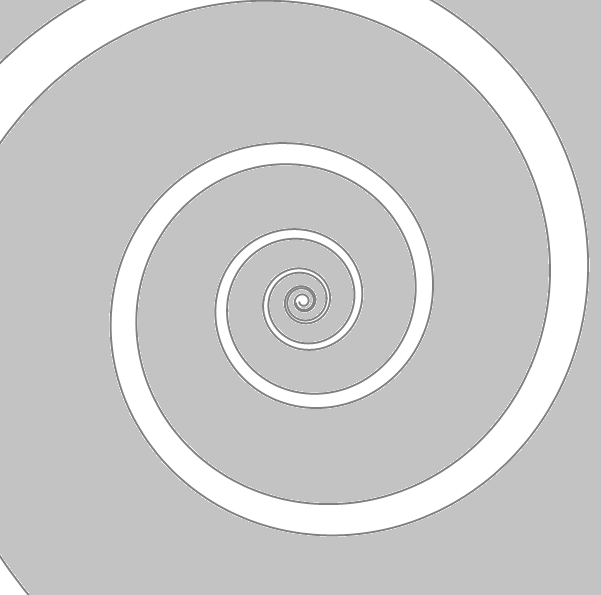}  
		\caption{Evolution of the vortex patch supported between two logarithmic spirals (gray region).} \label{fig:logspiral}
	\end{figure}
	
	\medskip
	
	Next, we consider the space $\calD(\bbS)$ of atomic measures on $\bbS$, i.e. $g \in \calD$ has a representation $g=\sum_{ j\ge0} I_j \dlt_{\tht_j}$ with $\sum_{j\ge0} |I_j|<\infty$. Assuming that $\tht_j$ are distinct, we set $\nrm{g}_{\calD} := \sum_{j\ge0} |I_j|$. As we shall discuss in more detail later, the following result gives a well-posedness for $h$ which are atomic measures, which corresponds to \textit{logarithmic vortex sheets} in $\bbR^{2}$. 
	
	\begin{theorem}[Logarithmic vortex sheets]\label{thm:wp-Dirac}
		Let $h_0 \in \calD(\bbS)$ satisfy $h_0= \sum_{ j\ge0} I_{j,0} \dlt_{\tht_{j,0}}$. Take some compactly supported family of mollifiers $\varphi^{\varep}(\tht)=\frac{1}{\varep}\varphi(\frac{\tht}{\varep})$. Then, there exists some $T>0$ such that for the sequence of mollified initial data $h_0^{\varep} := \varphi^\varep * h_0$, the corresponding unique global solutions $\left\{ h^{\varep}(t,\cdot)\right\}_{\varep>0} $ converge in the sense of measures to \begin{equation*}
			\begin{split}
				h(t,\cdot) = \sum_{ j\ge0} I_j(t) \dlt_{\tht_j(t)}
			\end{split}
		\end{equation*} for all $t\in[0,T)$. Here, $\left\{ I_j, \tht_j \right\}_{j\ge0 }$ is the unique local solution to the ODE system \begin{equation}\label{eq:ODE}
			\left\{
			\begin{aligned}
				\dot{I}_j(t) & = 2 \partial_{\theta} H(t,\tht_j(t)) I_j(t), \\
				\dot{\tht}_j(t) & = 2H(t,\tht_j(t)),
			\end{aligned}
			\right.
		\end{equation} with initial data $I_j(0)=I_{j,0}$ and $\tht_j(0)=\tht_{j,0}$, where $H$ is the unique Lipschitz solution to \eqref{eq:elliptic}. Specifically, we have that \begin{equation}\label{eq:Dirac-H-explicit}
			\begin{split}
				H(t,\tht_{j}(t)) = \sum_{\ell\ge0} I_{\ell}(t) K(\tht_{j}(t) - \tht_{\ell}(t) )
			\end{split}
		\end{equation} and \begin{equation}\label{eq:Dirac-Hprime-explicit}
			\begin{split}
				\partial_{\theta} H(t,\tht_j(t)) = \sum_{\ell\ge0} I_{\ell}(t) K'(\tht_{j}(t) - \tht_{\ell}(t) ), \qquad K'(0) := \lim_{\eps \to 0} \frac{K'(\eps) + K'(-\eps)}{2}
			\end{split}
		\end{equation} with $K$ being the fundamental solution to \eqref{eq:elliptic} given explicitly in \eqref{eq:cln K}. In this sense, the time-dependent Dirac measure $h(t,\cdot)$ is the unique solution to \eqref{eq:vort-evo}--\eqref{eq:elliptic} with initial data $h_0$ in $[0,T)$. 
	\end{theorem}

        Let us briefly comment on the statements: to begin with, the function $\partial_{\theta} H$ (which is the derivative of the solution $H$ to \eqref{eq:elliptic}) is actually not uniformly Lipschitz in $\bbS$. However, $\partial_{\theta} H$ in this case (meaning that $h$ in the right hand side of \eqref{eq:elliptic} is a linear combination of Dirac deltas) is infinitely smooth away from the support of the Dirac deltas, and in particular, one can prove existence and uniqueness of the solution to \eqref{eq:ODE} as long as none of the points $\tht_{j}$ collide with each other. (The situation is completely analogous to the case of point vortices evolving in a two-dimensional domain.) 
        
        Furthermore, we emphasize that the above convergence statement does not follow directly from some norm estimate (the main problem being $K \notin C^1(\bbS)$; this is why $K'(0)$ needs to be defined as the limit in \eqref{eq:Dirac-Hprime-explicit}) but instead one needs to use a crucial cancellation coming from the fact that the non-continuous part of $K'$ is \textit{odd}. In a sense, this proof is similar (although simpler) to the proof that the sequence of de-singularized vortex patch solutions to the 2D Euler equation converges to the solution of the point vortex system in the singular limit (\cite{MP}). In the statement, we could have regularized each Dirac delta by either a patch or some other $L^\infty$ functions as well. 
 
	\begin{remark} One may extend the above statement to get local existence in the broader class $\calM := \calD+L^1$. 
	\end{remark}

	The following proposition verifies that the solutions to the one-dimensional system \eqref{eq:vort-evo}--\eqref{eq:elliptic} obtained in such a class give rise to actual weak solutions to the two-dimensional Euler equations. 
	\begin{proposition}\label{prop:back-to-2D}
		Let $h \in C_*([0,T);\calM(\bbS))$ be a weak solution of \eqref{eq:vort-evo} with initial data $h_0$. Then, $(\omg, u)$ defined via \eqref{eq:vort-homog-spi}--\eqref{eq:vel-homog} provides an integral weak solution, in velocity form, to 2D Euler equations \eqref{eq:2DEuler-vort} with initial data $\omg_0$. For $h \in C([0,T);L^1(\bbS))$, we get moreover that $(\omg, u)$ is a weak solution in vorticity form.
	\end{proposition}

	\medskip
	
	\noindent \textbf{Long time dynamics and singularity formation.} Given the above local well-posedness results, it is natural to study the long-time dynamics with initial data either in $L^p$ or in $\calD$. It turns out that there is a \textit{monotone} quantity for solutions to \eqref{eq:vort-homog-spi}--\eqref{eq:vel-homog}, which is nothing but the local circulation: \begin{equation}\label{eq:entropy spirals}
		\begin{split}
			\Gmm(R) := \iint_{|x|\le R} \omg(t,x) dx = \frac{R^2}{2} \int h(t,\tht)d\tht.  
		\end{split}
	\end{equation} Unless $\omg$ is a constant, we have that $\Gmm(R)$ is \textit{strictly decreasing} (resp. increasing) for $\bt>0$ (resp. $\bt<0$), see Lemma \ref{lem:mon}. This is in stark contrast with the case of $0$-homogeneous vorticity studied in \cite{EJ,EMS} and one can see from the proof the specific nature of logarithmic spirals is reflected in the evolution of $\Gmm$. As an immediately corollary, we obtain that the only steady states to \eqref{eq:vort-homog-spi}--\eqref{eq:vel-homog} in $\calM$ are constants. This provides the basis for obtaining long time dynamics of solutions. 
	
	\medskip  
	
	In the case of bounded data, we can show long-time convergence to a constant steady state: 
	\begin{theorem}[Convergence for bounded data] \label{thm:conv}
		For $h_0 \in L^\infty(\bbS)$, there exist constants $\bfI_{\pm} = \bfI_{\pm}(h_0)$ satisfying $|\bfI_{\pm}|\le \nrm{h_0}_{L^\infty}$ such that the global-in-time solution $h(t,\cdot)$ corresponding to $h_0$ satisfies \begin{equation*}
			\begin{split}
				h(t,\cdot) \longrightarrow \bfI_{\pm},\qquad t \to \pm\infty		\end{split}
		\end{equation*} in $H^{-a}(\bbS)$ for any $a>0$. 
	\end{theorem}
	
	\begin{example}
		In the simple case when $h_0$ is the characteristic function of an interval, i.e. $h_0 =  \mathbf{1}_{J}$ for some $J\subset \bbS$, we have that $(\bfI_+,\bfI_-)=(0,1)$ and $(\bfI_+,\bfI_-)=(1,0)$, depending on whether $\bt>0$ or $\bt<0$, respectively. See Figure \ref{fig:logspiral} for an illustration: for $\bt<0$, if the initial vorticity in $\bbR^{2}$ is the patch supported in the gray region (Figure \ref{fig:logspiral}, center), then as $t\to\infty$, the support of vorticity occupies the entire $\bbR^2$ (Figure \ref{fig:logspiral}, right), while as $t\to-\infty$, the vorticity locally decays to 0 (Figure \ref{fig:logspiral}, left). 
	\end{example}

	\begin{theorem}[Trichotomy for $L^p$ data] \label{thm:tri}
		Consider $1\leq p<+\infty$,  $h_0\in L^p$  and $h(t,\cdot)\in L^p$ the unique local solution to \eqref{eq:vort-evo}-\eqref{eq:elliptic} on a maximal positive interval of existence $(0,T_*)$ with $T_*\in (0,+\infty]$. Then one of the following three scenarios must occur. 
		\begin{itemize}
			\item Either $T_*=+\infty$ and there exists $\bfI_+\in \mathbb{R}$ such that $h(t,\cdot)\underset{t\to \infty}{\longrightarrow} \bfI_+$ in $H^{-a}$ for $a>\max\left(0,\frac{1}{p}-\frac{1}{2}\right)$,
			\item the solution blows up in finite time: $T_*<+\infty$,
			\item or the solution blows up in infinite time: $\liminf_{t\to +\infty}\left\Vert h(t,\cdot)\right\Vert_{L^p}=+\infty$.
		\end{itemize}
	\end{theorem}

	\medskip
	
	In the case of Dirac initial data, we can provide a simple  criterion which guarantees finite time singularity formation. When there is no singularity formation in finite time, the solution must decay to 0. 
 
	\begin{theorem}[Singularity for Dirac measure data] \label{thm:blowup-Dirac}
		If the initial data $h_0 = \sum_{j=0}^{N-1} I_{j,0} \dlt_{\tht_j}$ is nonzero and satisfies $$\bt \sum_{j=0}^{N-1} I_{j,0} \le 0,$$ then the corresponding solution blows up in finite time; $\bt \sum_{j=0}^{N-1} I_j(t) \to -\infty$ as $t\to T^*$ for some $T^*<\infty$. 

        Furthermore, there is no finite time singularity formation if and only if the solution satisfies $\bt \sum_{j=0}^{N-1} I_j(t) \geq 0$ for all $t$ and $\bt \sum_{j=0}^{N-1} I_j(t) \to 0$ as $t\to +\infty$. This happens in particular when $\bt I_{j,0} > 0$ for all $j$.
	\end{theorem} 
	
	\begin{remark}
		Finite time singularity formation for logarithmic vortex sheets could seem paradoxical, in view of global well-posedness of bounded logarithmic vortex solutions and the convergence statement of Theorem \ref{thm:conv}. Singularity for the sheets can be interpreted as a form of strong instability for patches: initial data $h_{0}$ given by the characteristic function on an interval of length $0<\varep\ll 1$ will grow to become length $O(1)$ after an $O(1)$ time which is independent of $\varep$.
	\end{remark}

 \subsection{Background material}
	
	To put the consideration of logarithmic spiral vorticities and our main results into context, let us discuss various relevant topics for the incompressible Euler equations. 

    \medskip 
 
    \noindent\textbf{Vortex sheets supported on logarithmic spirals}. Spiraling behavior of fluid flows are quite frequently observed in turbulent jets and mixing layers at relatively large Reynolds number. While such vortex spirals are usually modeled by algebraic spirals in the applied literature, it seems that Prandtl in 1922 was the first one to suggest the possibility that logarithmic curves could be the profiles for such structures \cite{Prandtl61}. We refer to the interesting work \cite{Spiral} in which the authors inspect a variety of spiral flows observed in experiments and conclude that in many cases a logarithmic curve gives a better fit than an algebraic one. It turns out that, the case of one atom ($j=1$) in Theorem \ref{thm:wp-Dirac} exactly corresponds to the spiral suggested by Prandtl, up to a translation in time. Similarly, for $j\ge2$ with $j$-fold rotational symmetry, the logarithmic spiral vortex sheets obtained by Theorem \ref{thm:wp-Dirac} are simply the ones introduced later by Alexander \cite{Alexander71}. Indeed, these special solutions introduced by Prandtl and Alexander could be characterized by logarithmic vortex sheet solutions satisfying an additional \textit{self-similarity} with respect to time: $\omg(t,x)=t^{-1}\hat{\omg}(x)$ for some $\hat{\omg}$ (see the formulas and discussion in Appendix \ref{sec:self-similar}).  It was a highly non-trivial task to verify that such formulas give rise to actual weak solutions to the Euler equations (see for instance Saffman \cite[Section 8.3]{Saff}, Kambe \cite{Kam}, and Pullin \cite{Pull}), especially if one tries to apply the classical Birkhoff--Rott formulation \cite{Birkhoff,Rott}. 
    
     The mathematical proof of this was done in Elling--Gnann in the $m$-fold symmetric case with $m\ge3$ \cite{Elling-Gnann}, using special cancellation which is directly related with the well-posedness theory of 2d Euler under $m$-fold symmetry which we shall explain below. Without any symmetry hypothesis, the proof was done very recently by Cie{\'s}lak--Kokocki--O{\.z}a{\'n}ski in \cite{CKW21}. The same authors proved the existence of (a variety of) non-symmetric self similar logarithmic vortex spirals in \cite{CKW22}. A very nice review of the literature on logarithmic spirals and technical difficulties in treating those are given in \cite{CKW21,CKW22}. On the other hand, \cite{Elling-Gnann} contains many numerical computations which exhibit various bifurcation phenomena of non-symmetric spirals. The PDE approach proposed in this paper gives a unified framework in which all of the previous considerations, see Proposition \ref{prop:back-to-2D} and Appendix \ref{sec:self-similar},  can be treated in a much simpler fashion. Which from a mathematical point of view seems like the correct lens from which one should look upon logarithmic spirals for 2d perfect fluid.

	\medskip 
	
	\noindent\textbf{Long-time dynamics for Euler}.
	The global well-posedness of smooth enough solutions to \eqref{eq:2DEuler-vort} is now a well established fact. The long time behavior picture of such solutions is far from being complete. Indeed \eqref{eq:2DEuler-vort} is a non local, non linear transport equation modeling a perfect fluid. Both physical and numerical experiments suggest that most solutions relax in infinite time to simpler dynamics, i.e they experience a major contraction in phase space. This can be summarised by the following informal conjecture (see \cite{Sverakcours} and \cite{Shnirelman} respectively, \cite{EMS} and also the review articles \cite{ED-review,Shnirelman-turvey}) regarding the long time behavior of solutions to the 2d Euler equation:
	\begin{conjecture}\label{conj:loss compactness and compact orbits}
		\begin{enumerate}
			\item As $t\rightarrow\pm\infty,$ generic solutions experience loss of compactness.
			\item The (weak) limit set of generic solutions consists only of solutions lying on compact orbits.
		\end{enumerate}
	\end{conjecture}
	\noindent The only rigorous (and important) proofs of those conjectures are in the perturbative regimes around special steady states in the ground breaking work of Bedrossian and Masmoudi \cite{BedrossianMasmoudi} and later extensions by Ionescu and Jia \cite{IonescuJia1,IonescuJia2} and Masmoudi and Zhao \cite{MasmoudiZhao}. The only exception to this is the recent work \cite{EMS} where the conjecture is proven in full generality and in particular away from equilibrium but only for the subset of scale invariant m-fold symmetric $(m\geq4)$ solutions.

 \medskip 
	
	\noindent The results in this paper can thus be put in the larger picture of the long time behaviour of perfect fluids as follows. We give a rigorous construction of a special class of weak solutions of the Euler equation that is invariant under the flow, the class of \textit{logarithmic spirals} solutions first introduced in \cite{EJSVP1}, in Theorems \ref{thm:wp-Lp}, \ref{thm:gwp}, \ref{thm:wp-Dirac} and Proposition \ref{prop:back-to-2D}. Moreover we rigorously prove Conjecture \ref{conj:loss compactness and compact orbits} in this setting in full generality and again away from equilibrium in Theorems \ref{thm:conv}, \ref{thm:tri} and \ref{thm:blowup-Dirac}.
	
    \medskip

    \noindent Regarding the study of logarithmic spiral vortex sheets, our formulation based on \eqref{eq:vort-homog-spi} not only gives a very efficient way of accessing their dynamics, but also realize the vortex sheet evolution as the well-defined limit of more smooth objects, namely vorticities whose level sets are logarithmic spirals. The latter solutions can be completely smooth except at the origin. Furthermore, we emphasize that the spirals of Prandtl and Alexander are simply very specific solutions to the ODE system \eqref{eq:ODE} that we have obtained in this work, and this general approach provides a framework for studying the \textit{asymptotic stability} of self-similar singularity formation. To illustrate this, we recover some recent results from \cite{Elling-Gnann,CKW21,CKW22} on existence and bifurcation of self-similar logarithmic spiral vortex sheets using our formulation in Section \ref{sec:longtime}. 
    
	
    \medskip 
	
    \noindent\textbf{Symmetries and well-posedness of 2D Euler}. An unfortunate fact about non-trivial \textit{logarithmic spirals} solutions is that they cannot decay at spatial infinity which means they fail to belong to the standard well-posedness class for vorticity for 2d Euler $L^1(\mathbb{R}^2)\cap L^\infty(\mathbb{R}^2)$ given by Yudovich theory \cite{Y1}. Indeed the space of $L^1\cap L^\infty$ vorticity is stable under the flow of Euler and defines through the Biot--Savart law a log-Lipschitz velocity $u$ that in turn defines by the standard Osgood theory for ODEs a unique flow map 
	\[
	\frac{d}{dt}\Phi(t,x)=u\left(t,\Phi(t,x)\right) \text{ and } \Phi(0,x)=x.
	\]
	A key observation from \cite{EJ1} is that one can use the following discrete symmetry, which is preserved by the Euler flow, to drop the $L^1$ constraint on the vorticity. For $m\in\mathbb{N}$ a function
	$\omega:\mathbb{R}^2\rightarrow\mathbb{R}$ is said to be $m$-fold symmetric if $\omega(\mathcal{O}_{m}x)=\omega(x),$ for all $x\in\mathbb{R}^2,$ where $\mathcal{O}_m\in SO(2)$ is the matrix corresponding to a counterclockwise rotation by angle $\frac{2\pi}{m}.$ Indeed in \cite{EJ1} it is shown that for $\omega \in L^\infty_m(\mathbb{R}^2)$ with $m\geq 3$ then $\omega$ defines through the Biot--Savart law a log-Lipschitz velocity $u$. Thus $m-$fold bounded symmetric \textit{logarithmic spirals} are stable subset of solutions of a uniqueness class of solutions for the $2d$ Euler equations for which the spiraling motion induces an arrow of time and a strong relaxation mechanism in infinite time towards a completely homogenized fluid, Theorem \ref{thm:conv}. 

\medskip 

 \noindent Under the $m$-fold symmetric assumption with $m\ge 3$, the logarithmic spiraling dynamics can be realized as the dynamics at the origin of some compactly supported, finite energy solution solutions of the 2d Euler equation on $\mathbb{R}^2$ by the cut-off procedure given in the proof of Corollary 3.14 of \cite{EJ1}. Indeed, if $h_0\in W^{1,\infty}(\mathbb{S})$ is m-fold symmetric then for any $\omega^{2D}_0\in C^{0,1}_m\left(\mathbb{R}^2\right)$, there exists a unique global in time solution to the two dimensional Euler equation $\omega\in \mathring{C}^{0,1}$ such that 
 \[
 \omega^{2D}(t)=\omega(t)-h(t)\in C^{0,1}\left(\mathbb{R}^2\right) \text{ for all time},
 \]
and $h$ is the solution of \eqref{eq:vort-evo}--\eqref{eq:elliptic} with initial data $h_0$. We note that $\omega^{2D}_0$ can be chosen in such a fashion that $\omega_0$ is compactly supported.

	\medskip 
	
	\noindent\textbf{Dynamics of 0-homogeneous vorticity}.  
	The special case $\bt=0$ corresponds to the system for 0-homogeneous vorticity studied in \cite{EJ1,EMS}. A first major difference between between the $\bt=0$ and $\bt\neq 0$ cases is that for $\bt=0$ it is necessary to assume $m-$fold $(m\geq 3)$ symmetry on vorticity in order to ensure the well posedness of \eqref{eq:elliptic}. Within symmetry, the same techniques used here can be used to extend the local well-posedness results for bounded $m-$fold symmetric scale invariant solutions in \cite{EJ1} and get analogous results to Theorems \ref{thm:wp-Lp}, \ref{thm:gwp} and \ref{thm:wp-Dirac}. In \cite{EMS}, with the help of a monotone quantity quantifying the number of particles exiting the origin, it was possible to show that $BV$ $m-$fold $(m\geq 4)$ symmetric data relax in infinite time to states with finitely many jumps. This entropy found in \cite{EMS} is much weaker than the monotonicity of the local circulation  \eqref{eq:entropy spirals} exhibited here; indeed for $\beta=0$, \eqref{eq:entropy spirals} is conserved in time. This weaker entropy thus leaves the room for steady states that are not identically constant and cannot handle data that is not in the closure of $BV$  in $L^\infty$, let alone in $L^p$ or $\mathcal{D}(\mathbb{S})$.

 \medskip 

 \noindent Finally, on any finite time interval where both the $m-$fold symmetric $m\geq 3$ solutions of the 0-homogeneous equations, \eqref{eq:vort-evo}  with $\beta=0$, $h^0(t)$ and the logarithmic spiraling equations $h^\beta(t)$ start from the same initial data $h^0(0)=h^\beta(0)=h_0\in L^p, 1 \leq p <+\infty$ then $h^\beta(t) \to h_0(t)$ in $L^p$ when $\beta$ goes to $0$. For $h_0 \in L^\infty$ we get convergence in the weak topology and for $h_0 \in \mathcal{D}(\mathbb{S})$ we get convergence in the sense of measures. Indeed this follows from the observation that the kernel associated to \eqref{eq:elliptic} $K_\beta^m$ converges in $W^{1,\infty}$ to $K^m$ when $\beta $ goes to 0, which in turn follows from the Cauchy--Lipschitz theorem with parameters applied to $K^m_\beta-K^m$ and the observation from Remark \ref{rem:derv ker 0} that $K_\beta^m(0)$ and ${K_\beta^m}'(0)$ converge to $K^m(0)$ and ${K^m}'(0)$, respectively.

 \medskip 

 \noindent One may consider the class of $\alpha$-homogeneous vorticities for $\alp\in \bbR$, which satisfies $\omg(\lmb x)= \lmb^\alp \omg(x)$ for all $\lmb>0$. Existence and nonexistence results (depending on the range of $\alp$) are given in \cite{Abe}.

\begin{remark}
    We would like to clarify the difference of the logarithmic spiral solutions considered here with such solutions for the 2d incompressible Navier--Stokes equations (see for instance \cite{GuWi,GuWi2,Sve,Hamel,Landau}). In the latter case, the ansatz for the vorticity is given by (see equation (6) of \cite{GuWi}) $$\omega(r,\theta)=\lmb^{2} \omega(\lmb r, \theta + \bt \ln \lmb), $$ using our notation. Note that the factor $\lmb^{2}$ reflects the only allowed self-similar scaling of the Navier--Stokes equations; the velocity and vorticity decays with rate $r^{-1}$ and $r^{-2}$ at infinity, respectively. In the paper \cite{GuWi}, Guillod and Wittwer classified steady solutions to the 2d Navier--Stokes equations in $\mathbb{R}^{2}\backslash \{0\}$ satisfying the above symmetry. Among others, these solutions demonstrate the complicated nature of the set of steady states of the 2d Navier--Stokes equations; we refer the interested readers to the discussion in \cite{GuWi}. 
\end{remark}

\subsection{Further questions} Many interesting problems remain open for logarithmic vortex spirals. One problem is to consider the inviscid limit of Navier--Stokes solutions towards logarithmic spiral vortex sheets and prove convergence. Furthermore, it is an important problem to understand stability of the logarithmic vortex spirals within 2d Euler equations. This was already considered in a recent work \cite{CKW24}, where it was proved that at least for a specific class of initial perturbations, the perturbation grows only polynomial in time and not exponential. This issue certainly deserves further investigations.

\subsection{Organization of the paper} The rest of the paper is organized as follows. In Section \ref{sec:prelim}, we obtain some simple properties of the kernel of the elliptic problem \eqref{eq:elliptic}. Then, the main well-posedness results are proved in Section \ref{sec:wp}. Section \ref{sec:longtime} contains results pertaining to the long time dynamics of solutions, as well as some case studies of Dirac deltas. In particular, we recover the existence and bifurcation of symmetric and non-symmetric self similar logarithmic spiral vortex sheets. The explicit form of the kernel is given in the Appendix. 
	
\subsection*{Acknowledgments}
		IJ has been supported  by the Samsung Science and Technology Foundation under Project Number SSTF-BA2002-04. AS acknowledges funding from the NSF grants DMS-2043024 and DMS-2124748. We are very grateful for Theodore Drivas and Tarek Elgindi for various helpful discussions and suggesting several references. In particular, Tarek Elgindi suggested the proof of global well-posedness under the assumption \eqref{eq:Lp-log}. We would like to thank Pr. Cie{\'s}lak, Pr. O{\.z}a{\'n}ski, Pr. Constantin, Pr. Ionescu and Pr. Levermore for insightful comments and discussions on early presentations of this work. Furthermore, let us mention that the current paper was strongly inspired by works of Elling--Gnann and Cie{\'s}lak--Kokocki--O{\.z}a{\'n}ski on logarithmic spiral vortex sheets. Lastly, we thank the anonymous referees for their careful reading of the manuscript and providing various suggestions and pointing out several typos, which have significantly improved the paper. 
	
	\section{Preliminaries}\label{sec:prelim}
	
	\subsection{Properties of the kernel}\label{subsec:elliptic}
	
	In this section, we deal with the elliptic equation \eqref{eq:elliptic}. For any $h \in L^p$ with $p\ge1$, the unique solution $H \in W^{2,p}$ is given by the convolution $H=K*h$, where the kernel $K$ is defined by the unique solution to the ODE 
	\begin{equation}\label{eq:K}
 \begin{split}
			4K - 4\bt K' + (1+\bt^2)K'' = 0, \quad 0<\tht<2\pi, \\
			K(0) = K(2\pi), \quad K'(2\pi)-K'(0) := \lim_{\epsilon\to 0}K'(2\pi-\epsilon)-K'(\epsilon)=\frac{1}{1+\beta^2}. 
    \end{split}
	\end{equation} In the Appendix, we derive the explicit form of the kernel based on Fourier series. We record a few simple properties of $K$ below. 
	\begin{lemma} We have the formula
		\begin{equation}\label{eq:K0}
			\begin{split}
				\int  (K'(\tht ))^2 d\tht + \int \frac{4}{1+\bt^2} K^2(\tht ) d\tht  = \frac{1}{1+\bt^2}K(0).  
			\end{split}
		\end{equation} 
	\end{lemma}
	\begin{proof} We start with 
		\begin{equation*}
			\begin{split}
				&\int K'(\tht )K'(\tht ) d\tht = \lim_{\varepsilon\to0} \int_{ |\tht |>\varepsilon}   K'(\tht )K'(\tht ) d\tht  \\
				&\qquad =  - \lim_{\varepsilon\to0} \int_{ |\tht |>\varepsilon}  K''(\tht )K(\tht ) d\tht + \lim_{\varepsilon\to0} (K'(\varepsilon)K(\varepsilon) - K'(-\varepsilon)K(-\varepsilon)). 
			\end{split}
		\end{equation*} We note that \begin{equation*}
			\begin{split}
				\lim_{\varepsilon\to0} (K'(\varepsilon)K(\varepsilon) - K'(-\varepsilon)K(-\varepsilon)) = \frac{1}{1+\bt^2}K(0).  
			\end{split}
		\end{equation*} Next, \begin{equation*}
			\begin{split}
				&\lim_{\varepsilon\to0} \int_{ |\tht |>\varepsilon}  K''(\tht )K(\tht ) d\tht  = \lim_{\varepsilon\to0} \int_{ |\tht |>\varepsilon} \frac{1}{1+\bt^2} \left( -4K(\tht ) + 4\bt K'(\tht ) \right)K(\tht ) d\tht \\
				&\qquad = \lim_{\varepsilon\to0} \int_{ |\tht |>\varepsilon} -\frac{4}{1+\bt^2} K^2(\tht ) + \frac{2\bt}{1+\bt^2} \rd_\tht ( K^2(\tht ) ) d\tht = \int -\frac{4}{1+\bt^2} K^2(\tht ) d\tht . 
			\end{split}
		\end{equation*} This gives \eqref{eq:K0}. 
	\end{proof}
	
	\begin{lemma} We have 
		\begin{equation}\label{eq:Kp0}
			\begin{split}
				K'(0) = - 4\bt \int (K'(\tht))^2 d\tht. 
			\end{split}
		\end{equation} For any $0 \ne \alp \in \bbS$, \begin{equation}\label{eq:Kpa}
			\begin{split}
				\left( K'(\alp) + K'(-\alp) \right)= -4\bt \int K'(\tht) \left( K'(\tht+\alp) + K'(\tht-\alp) \right) d\tht. 
			\end{split}
		\end{equation}
	\end{lemma}
	\begin{proof}
		We multiply both sides of \eqref{eq:K} by $K'$ in the region $|\tht|>\varepsilon$ to obtain \begin{equation*}
			\begin{split}
				\int_{ |\tht |>\varepsilon} KK' d\tht - 4\bt \int_{ |\tht |>\varepsilon} (K')^2d\tht + (1+\bt^2) \int_{ |\tht |>\varepsilon} K''K' d\tht = 0. 
			\end{split}
		\end{equation*} Using that \begin{equation*}
			\begin{split}
				\int_{ |\tht |>\varepsilon} KK' d\tht = \frac12 \left( K^2(\varepsilon)-K^2(-\varepsilon) \right) \to 0
			\end{split}
		\end{equation*} and \begin{equation*}
			\begin{split}
				\int_{ |\tht |>\varepsilon} K''K' d\tht = \frac12 \left(  (K')^2(\varepsilon) - (K')^2(-\varepsilon) \right) = \frac{1}{1+\bt^2} \frac{ K'(\varepsilon) + K'(-\varepsilon) }{2} \to \frac{1}{1+\bt^2} K'(0)
			\end{split}
		\end{equation*} as $\varepsilon\to0$, we conclude \eqref{eq:Kp0}. The proof of \eqref{eq:Kpa} is similar. We multiply both sides of \eqref{eq:K} by $K'(\tht+\alp)+K'(\tht-\alp)$ and integrate in the region $A:= \bbS\backslash ( \{ |\tht-\alp|<\varepsilon \} \cup \{ |\tht+\alp|<\varepsilon \} )$. Then \begin{equation*}
			\begin{split}
				\int_A K(\tht) (K'(\tht+\alp)+K'(\tht-\alp)) d\tht  = -\int_A K'(\tht) ( K(\tht+\alp) + K(\tht-\alp) ) d\tht + O(\varepsilon),
			\end{split}
		\end{equation*} which shows that \begin{equation*}
			\begin{split}
				2\int_A K(\tht) (K'(\tht+\alp)+K'(\tht-\alp)) d\tht  = O(\varepsilon). 
			\end{split}
		\end{equation*} Similarly, as $\varepsilon\to0$, it can be shown using integration by parts that \begin{equation*}
			\begin{split}
				(1+\bt^2)\int_{A} K''(\tht)(K'(\tht+\alp)+K'(\tht-\alp)) d\tht \rightarrow K'(\alp) + K'(-\alp). 
			\end{split}
		\end{equation*} This finishes the proof. 
	\end{proof}
	
	\subsection{Monotonicity}
	
	\begin{lemma}\label{lem:mon}
		Any $L^p$ solution to \eqref{eq:vort-evo} satisfies \begin{equation}\label{eq:mon}
			\begin{split}
				\frac{d}{dt} \int h d\tht = -8\bt \int (H')^2 d\tht. 
			\end{split}
		\end{equation} For $h = \sum_{i=1}^{N} I_i(t) \dlt(\tht-\tht_i(t)), $ \begin{equation}\label{eq:mon-di}
			\begin{split}
				\frac{d}{dt}  \sum_{i=1}^{N} I_{i}  = - 8\bt \int (H')^2 d\tht
			\end{split}
		\end{equation} where \begin{equation*}
			\begin{split}
				H'(t,\tht) = \sum_{i=1}^{N} I_i(t) K'(\tht-\tht_i(t)).
			\end{split}
		\end{equation*}
	\end{lemma}
	\begin{proof}
		Assuming that $h$ is smooth, 
		\begin{equation*}
			\begin{split}
				\frac{d}{dt} \int h d\tht & = -2 \int H\rd_\tht h d\tht = 2\int  H' ( 4H - 4\bt  H' + (1+\bt^2) H'' ) d\tht \\
				& = -8\bt \int (H'')^2 d\tht + \int \rd_\tht \left(  4H^2 + (1+\bt^2) (H')^2  \right) d\tht  = -8\bt \int (H'')^2 d\tht. 
			\end{split}
		\end{equation*} The case of $h\in L^p$ follows from an approximation argument. In the case of Dirac deltas, we have 
		\begin{equation*}
			\begin{split}
				\frac12 \frac{d}{dt} \sum_{i=1}^{N} I_{i} = \sum_{i,j} \int I_i(t) I_j(t) K'(\tht_i(t)-\tht_j(t)). 
			\end{split}
		\end{equation*} On the other hand, \begin{equation*}
			\begin{split}
				-4\bt \int (\partial^2_{\theta} H)^2 d\tht & = -4\bt \sum_{i,j} \int I_i(t)I_j(t) K'(\tht-\tht_j(t))K'(\tht-\tht_i(t)) d\tht . 
			\end{split}
		\end{equation*} Then \eqref{eq:mon-di} follows from \eqref{eq:Kp0} and \eqref{eq:Kpa}. 
	\end{proof}

	\section{Well-posedness issues}\label{sec:wp}
	
	\subsection{Proof of well-posedness results}

In this section, we give the proof of existence and uniqueness. We emphasize that the uniqueness statement we prove is uniqueness only among the solutions  in the class of logarithmic spiral vortices.

	\begin{proof}[Proof of Theorem \ref{thm:wp-Lp}]
		
		We divide the proof into three parts. 
		
		\medskip
		
		\noindent \textbf{1. Local existence in $L^p$}. We obtain a priori estimates in $L^p$ for any $p\ge1$. For this, assume that we are given a smooth solution $h, H$ to \eqref{eq:vort-evo}--\eqref{eq:elliptic}. Then, \begin{equation}\label{eq:drv lp norm}
			\begin{split}
				\frac{d}{dt} \int |h|^p d\tht &= -p \int |h|^{p-1} (2H h') \sgn(h) d\tht  = \int 2H' |h|^p d\tht \le 2\nrm{H'}_{L^\infty} \int |h|^p d\tht. 
			\end{split}
		\end{equation} Using that $\nrm{H'}_{L^\infty} \le C\nrm{h}_{L^p}$ holds for any $p\ge1$, \begin{equation*}
			\begin{split}
				\frac{d}{dt} \nrm{h}_{L^p} \le \frac{C}{p} \nrm{h}_{L^p}^2. 
			\end{split}
		\end{equation*} In the case $p=\infty$, we obtain \begin{equation*}
			\begin{split}
				\frac{d}{dt} \nrm{h}_{L^\infty} = 0. 
			\end{split}
		\end{equation*} This gives that $\nrm{h}_{L^\infty([0,T^*];L^p)} \le 2\nrm{h_0}_{L^p}$ for $T^*>0$ depending only on $p$ and $\nrm{h_0}_{L^p}$. Based on this a priori estimate, proving the existence of an $L^p$ solution can be done by the method of mollification. Given any initial data $h_0 \in L^p$, consider the sequence of mollified data $h_0^{\varep}$ converging to $h_0$ in $L^p$. For each $h_0^\varep$, one can construct a corresponding local smooth solution $(h^\varep,H^\varep)$ to  \eqref{eq:vort-evo}--\eqref{eq:elliptic}, for example by an iteration scheme. The sequence of solutions $h^\varep$ remains smooth in the time interval $[0,T^*]$ and satisfies the uniform bound $\nrm{h^\varep}_{L^\infty([0,T^*];L^p)} \le 2\nrm{h_0}_{L^p}$. Appealing to the Aubin--Lions lemma, this gives a weak-$L^p$ limit $h \in L^\infty([0,T^*];L^p)$ as $\varep\to0$, by passing to a subsequence if necessary. The corresponding limit  $H^\varep\to H$ is strong in $W^{1,p}$, and this shows that $(h,H)$ gives a weak solution to \eqref{eq:vort-evo}--\eqref{eq:elliptic}. (This argument is parallel, and only easier, compared to the well-known proof of existence for $L^p$ vorticity weak solutions to 2D incompressible Euler equations, see \cite{MB}.)

		\medskip
		
		\noindent \textbf{2. Uniqueness in $L^1$}.  We prove that there is at most one solution in the class $L^\infty([0,T];L^1)$ of \textit{logarithmic spiraling solutions}. The case $p>1$ is only easier. Given an initial datum $h_0 \in L^1$, we assume that there are two associated solutions $h$ and $\tilde{h}$ belonging to $L^\infty([0,T];L^1)$ for some $T>0$. We denote $H = \bfK h$ and $\tilde{H} = \bfK\tilde{h}$. By taking $\bfK$ to the equation for $h$, we may derive the evolution equations satisfied by $H$ and $H'$: \begin{equation}\label{eq:H}
			\begin{split}
				\rd_tH + 2H \rd_\tht H = -\bfK(8\bt - 3(1+\bt^2)\rd_\tht)\left[ (H')^2  \right]
			\end{split}
		\end{equation}\begin{equation}\label{eq:H-prime}
			\begin{split}
				\rd_t H' + 2H \rd_\tht H' = (H')^2 - 4 \bfK (3-\bt\rd_\tht)\left[ (H')^2 \right]. 
			\end{split}
		\end{equation} Denoting $D = (H-\tilde{H}) $, we see that it satisfies \begin{equation}\label{eq:D}
			\begin{split}
				\rd_t D + 2H \rd_\tht D + 2D\tilde{H}' =   - \bfK(8\bt -3(1+\bt^2)\rd_\tht) \left[  (H'+ \tilde{H}')D' \right].
			\end{split}
		\end{equation} Here, $H', \tilde{H}' \in L^\infty([0,T];W^{1,1}) \subset L^\infty([0,T];C^0)$. The operator $  - \bfK(8\bt -3(1+\bt^2)\rd_\tht) $ is convolution type with a bounded kernel. We consider the estimate of $\bfK\rd_\tht (g D')$, where $g$ is a $W^{1,1}$ function: \begin{equation*}
			\begin{split}
				\bfK\rd_\tht (g D')(\tht) = \int K'(\tht-\tht') g(\tht') D'(\tht') d\tht'. 
			\end{split}
		\end{equation*} Note that for any $\tht$, $K'(\tht-\tht') g(\tht')$ is differentiable in the sense of distributions with respect to $\tht'$, with \begin{equation*}
			\begin{split}
				\rd_{\tht'} \left( K'(\tht-\tht') g(\tht') \right) = K'(\tht-\tht') g'(\tht') - \frac{\dlt(\tht-\tht')}{1+\bt^2}g(\tht') + \frac{4\bt}{1+\bt^2}K'(\tht-\tht')g(\tht') - \frac{4}{1+\bt^2}K(\tht-\tht') g(\tht'). 
			\end{split}
		\end{equation*} Since $g$ and $D$ are continuous functions, this allows us to rewrite \begin{equation*}
			\begin{split}
				\bfK\rd_\tht (g D')(\tht) = \int A(\tht-\tht')g'(\tht')D(\tht')d\tht' + \frac{1}{1+\bt^2} g(\tht) D(\tht),
			\end{split}
		\end{equation*} where $A$ is a bounded function. In particular, we obtain the estimate \begin{equation*}
			\begin{split}
				\nrm{ \bfK\rd_\tht (g D') }_{L^2} \le C\nrm{g}_{L^\infty} \nrm{D}_{L^2}. 
			\end{split}
		\end{equation*} Similarly, \begin{equation*}
			\begin{split}
				\nrm{ \bfK (g D') }_{L^2} \le C\nrm{g}_{L^\infty} \nrm{D}_{L^2}. 
			\end{split}
		\end{equation*} We then estimate, by multiplying \eqref{eq:D} against $D$ and integrating, \begin{equation*}
			\begin{split}
				\frac{d}{dt} \nrm{D}_{L^2}^2 &\le C( \nrm{H'}_{L^\infty} + \nrm{\tilde{H}}_{L^\infty} ) \nrm{D}_{L^2}^2 + \nrm{     \bfK(8\bt -3(1+\bt^2)\rd_\tht) ( (H'+ \tilde{H}')D') }_{L^2} \nrm{D}_{L^2} \\
				& \le C( \nrm{h}_{L^1} + \nrm{\tilde{h}}_{L^1} ) \nrm{D}_{L^2}^2. 
			\end{split}
		\end{equation*} Integrating in time, we obtain \begin{equation*}
			\begin{split}
				\nrm{D}_{L^2}^2 \le \nrm{D_0}_{L^2}^2 \exp\left(  C( \nrm{h}_{L^\infty_t L^1} + \nrm{\tilde{h}}_{L^\infty_t L^1} ) \right)
			\end{split}
		\end{equation*} which gives uniqueness.

		\medskip
		
		\noindent \textbf{3. Blow-up criterion and global existence in $L^\infty$}. From the $L^p$ a priori estimate, we have that \begin{equation*}
			\begin{split}
				\nrm{h(t,\cdot)}_{L^p} \le \nrm{h_0}_{L^p} \exp\left( \frac{2}{p} \int_0^t \nrm{H'(\tau,\cdot)}_{L^\infty} d\tau \right). 
			\end{split}
		\end{equation*} Therefore, the local unique $L^p$ solution blows up in $L^p$ at $T^*$ if and only if \begin{equation}\label{eq:BKM}
			\begin{split}
				\int_0^{T^*}\nrm{H'(\tau,\cdot)}_{L^\infty} d\tau=\infty.
			\end{split}
		\end{equation} Then for each time, $\nrm{H'(\tau,\cdot)}_{L^\infty}$ is equivalent to $\nrm{h(\tau,\cdot)}_{L^1}$, from which the claimed blow-up criterion follows. In the case $p=\infty$, we have that as long as the solution exists, $\nrm{h(t,\cdot)}_{L^\infty} = \nrm{h_0}_{L^\infty}$. In turn, this gives $\nrm{H'(t,\cdot)}_{L^\infty} \le C\nrm{h_0}_{L^\infty}$. This implies that the $L^\infty$ solutions are global in time, in view of \eqref{eq:BKM}. 
	\end{proof}
	
	\subsection{Global well-posedness}

 \begin{proof}[Proof of Theorem \ref{thm:gwp}]
     The goal is to show the propagation of the constraint $\lim_{p\rightarrow+\infty}\frac{\left \Vert h(t,\cdot)\right\Vert_{L^p}}{p}<+\infty$ for all $t$. Starting back from \eqref{eq:drv lp norm} we get for $1\leq p<+\infty$
	\[
	\frac{d}{dt}\left\Vert h\right\Vert_{L^{p}}\leq \frac{2}{p}\left\Vert H'\right\Vert_{L^{\infty}}\left\Vert h\right\Vert_{L^{p}}.
	\]
	By the Gronwall lemma we get
	\[
	\left\Vert h(t,\cdot)\right\Vert_{L^{p}}\leq e^{\frac{2\int_0^t\left\Vert H'\right\Vert_{L^{\infty}}}{p}}\left\Vert h_0\right\Vert_{L^{p}},
	\]
	thus for $p$ such that $\frac{2\int_0^t\left\Vert H'\right\Vert_{L^{\infty}}}{p}\leq \ln(2)$ we have 
	\[
	\left\Vert h(t,\cdot)\right\Vert_{L^{p}}\leq 2 \left\Vert h_0\right\Vert_{L^{p}}\leq 2C_0 p.
	\] 
	Now, so long as the solution exists, choose \[p=\max\{1,4\int_0^t \left\Vert H'\right\Vert_{L^\infty}\}.\] Note that \[\left\Vert H' \right\Vert_{L^\infty}\leq C \left\Vert h\right\Vert_{L^1},\qquad \left\Vert h\right\Vert_{L^1}\leq C\left\Vert h\right\Vert_{L^{p}},\] for any $p\geq 1$ and for $C>0$ universal. 
	It follows then that 
	\[\left\Vert h(t,\cdot)\right\Vert_{L^1}\leq C \max\{1,\int_{0}^t\left\Vert h\right\Vert_{L^1}\}.\]
	It now follows by the Gronwall Lemma that 
	\[\left\Vert h(t,\cdot)\right\Vert_{L^1}\leq C\exp(Ct).\]
	which gives the desired result on the bound on the constraint and global existence.
 \end{proof}

	\subsection{Convergence to logarithmic spiral sheets}\label{sec:cvg moll}
	
	\begin{proof}[Proof of Theorem \ref{thm:wp-Dirac}]
		Let $h_0 = \sum_{j\ge0} I_{j,0}\dlt_{\tht_{j,0}}$ satisfy the assumptions of Theorem \ref{thm:wp-Dirac}. Now let us note that under those hypothesis that there exists $\eta\ll1 $ such $H'$ is smooth on all of the intevals $\left(\theta_j(0)-\eta,\theta_j(0)+\eta\right)$ and thus the standard Cauchy-Lipschitz theorem can be applied to guarantee the existence of a unique $h$ corresponding to the local in time solution given by \eqref{eq:ODE}. We let $h^\varep$ as the sequence of smooth and global-in-time solutions with mollified initial data $h^\varep_0 := \varphi^\varep * h_0$.

		To begin with, note that as long as the solution to \eqref{eq:ODE} does not blow up, we have $\tht_{j}(t) \ne \tht_{k}(t)$ for $j\ne k$. We fix such a time interval $[0,T]$ and then for each $j$ we can take some $\eta_{j}>0$ such that $h(t,\cdot) \mathbf{1}_{A_{j}(t)} = I_{j}(t) \dlt_{\tht_j(t)}$, where $A_{j}(t):= (\tht_{j}(t) - \eta_{j}, \tht_{j}(t)+\eta_{j})$ is the open interval.  
		
		Since the sequence of data $h_0^{\varep}$ is uniformly $L^1$, from the $L^1$ estimate given in the previous section, we have that on $[0,T]$ (by taking $T>0$ smaller if necessary) $h^\varep(t,\cdot)$ is uniformly $L^1$ in $t$ and $\varep$, and $H^{\varep}(t,\cdot)$ is uniformly Lipschitz continuous in $t$ and $\varep$. (Here, all the relevant norms are bounded in terms of $\nrm{h_0}_{\calD}$.) 
		
		Note that $\supp(h_0^{\varep}) \subset B( \supp(h_0), C\varep )$ where we define for simplicity $B(K,\eta)$ as the $\eta$-neighborhood of some set $K$. For each fixed $j$, by taking $\varep>0$ sufficiently small, we can ensure that $\supp(h_0^{\varep})$ has a unique connected component intersecting $A_{j}(t=0)$. By continuity, there is a small time interval on which $\supp(h^\varep(t,\cdot))$ still has this property for all sufficiently small $\varep>0$. Then, the quantities $\{ \tht_{j}^\varep(t), I_{j}^{\varep}(t) \}$ are well-defined by the following equations:  \begin{equation*}
			\begin{split}
				\tht_{j}^{\varep}(t) \int_{A_{j}(t)} h^\varep(t,\tht) d\tht = \int_{A_{j}(t)} \tht h^\varep(t,\tht) d\tht , \quad \mbox{and} \quad I_{j}^{\varep}(t) =  \int_{A_{j}(t)} h^\varep(t,\tht) d\tht. 
			\end{split}
		\end{equation*} Note that uniform Lipschitz continuity of $H^\varep$ implies that the length of $\supp(h^\varep)\cap A_{j} \le C\varep$ with $C$ depending only on the Lipschitz norm. This in particular guarantees that \begin{equation*}
			\begin{split}
				\int_{A_j(t)} |\tht - \tht_{j}^{\varep}(t)| h^\varep(t,\tht) d\tht \le C\varep I_{j}^{\varep}(t),
			\end{split}
		\end{equation*} simply because $\tht_{j}^{\varep}(t) \in \supp(h^\varep(t,\cdot)) \cap A_{j}$. 
		Now, differentiating in time the above relations (and using that $h^\varep$ is smooth and vanishes on $\partial A_{j}$), \begin{equation}\label{eq:limit-1}
			\begin{split}
				\frac{d}{dt} I_{j}^{\varep}(t) = \int_{A_{j}(t)} 2 (H')^\varep(t,\tht) h^\varep(t,\tht) d\tht, 
			\end{split}
		\end{equation}
		\begin{equation}\label{eq:limit-2}
			\begin{split}
				\tht_{j}^{\varep}(t)\frac{d}{dt} I_{j}^{\varep}(t) +I_{j}^{\varep}(t)  \frac{d}{dt} \tht_{j}^{\varep}(t)= \int_{A_{j}(t)} 2 H^\varep(t,\tht) h^\varep(t,\tht) d\tht + \int_{A_{j}(t)} 2\tht (H')^\varep(t,\tht) h^\varep(t,\tht) d\tht .
			\end{split}
		\end{equation} We use \eqref{eq:limit-1} to rewrite \eqref{eq:limit-2} as follows: \begin{equation*}
			\begin{split}
				\frac{d}{dt} \tht_{j}^{\varep}(t) & = 2H^\varep(t, \tht_{j}^{\varep}(t) ) + \frac{2}{I_{j}^{\varep}(t)   } \int_{A_{j}(t)} ( H^\varep(t,\tht) - H^\varep(t,\tht_{j}^{\varep}(t) ) ) h^\varep(t,\tht) d\tht \\
				&\qquad  + \frac{2}{I_{j}^{\varep}(t) }\int_{A_{j}(t)} (\tht - \tht^\varep_j(t))(H')^\varep(t,\tht) h^\varep(t,\tht) d\tht. 
			\end{split}
		\end{equation*} Estimating \begin{equation*}
			\begin{split}
				\left| \int_{A_{j}(t)} ( H^\varep(t,\tht) - H^\varep(t,\tht_{j}^{\varep}(t) ) ) h^\varep(t,\tht) d\tht\right| \le \nrm{ (H')^\varep }_{L^\infty}  \int_{A_{j}(t)} |\tht -\tht_{j}^{\varep}(t) |h^\varep(t,\tht) d\tht \le C\varep I_j^{\varep}(t)
			\end{split}
		\end{equation*} and similarly the other term, we obtain that \begin{equation}\label{eq:ODE-tht-varep}
			\begin{split}
				\left|  \frac{d}{dt} \tht_{j}^{\varep}(t) - 2H^{\varep}(t, \tht_{j}^{\varep}(t) ) \right| \le C\varep. 
			\end{split}
		\end{equation} 
		
		Let us now derive a similar estimate for \eqref{eq:limit-1}. To this end, we first decompose $H^\varep = H^\varep_j + H^\varep_{\ne j}$; $H^{\varep}_j$ is simply defined as the solution to \eqref{eq:elliptic} with right hand side $h^\varep_j := h^\varep \mathbf{1}_{A_{j}}$. We write \begin{equation*}
			\begin{split}
				\int_{A_{j}(t)} 2 (H')^\varep(t,\tht) h^\varep(t,\tht) d\tht = \int_{A_{j}(t)} 2 (H')^\varep_{j}(t,\tht) h^\varep(t,\tht) d\tht +\int_{A_{j}(t)} 2 (H')^\varep_{\ne j}(t,\tht) h^\varep(t,\tht) d\tht . 
			\end{split}
		\end{equation*} Then note that $H^\varep_{\ne j}$ is smooth on $A_{j}$ from the support property, and this gives by writing $(H')^\varep_{\ne j}(t,\tht) = (H')^\varep_{\ne j}(t,\tht_j^\varep) + O(|\tht - \tht_j^\varep|)$  \begin{equation*}
			\begin{split}
				\int_{A_{j}(t)} 2 (H')^\varep_{\ne j}(t,\tht) h^\varep(t,\tht) d\tht  = 2 (H')^\varep_{\ne j}(t,\tht^\varep_j(t)) I^{\varep}_j(t) + O(\varep). 
			\end{split}
		\end{equation*} Next, we have that \begin{equation}\label{eq:sym}
			\begin{split}
				\int_{A_{j}(t)} 2 (H')^\varep_{j}(t,\tht) h^\varep(t,\tht) d\tht = \int_{\bbS} 2 (\rd_\tht K*h^\varep_j) (\tht) h^\varep_j(\tht) d\tht 
			\end{split}
		\end{equation} where $K$ is the kernel for the elliptic problem \eqref{eq:elliptic} (see Appendix for its explicit form). We consider $K$ as defined on $[-\pi,\pi]$, and then we can decompose $K= K^e + K^o$ where $K^e$ and $K^o$ are even and odd parts of $K$ around $\tht = 0$, respectively. Using the ODE satisfied by $K$, we can derive the relation \begin{equation*}
			\begin{split}
				(1+\bt^2) (K^o)'' - 4\bt (K^e)' + K^o = 0, 
			\end{split}
		\end{equation*} and since $(K^e)' \in L^\infty(\bbS)$, we have that $K^o \in W^{2,\infty}(\bbS)$ (we only need it to be strictly better than Lipschitz). Alternatively, $K^o \in W^{2,\infty}(\bbS)$ can be checked using the explicit formula \eqref{eq:cln K} given in the Appendix. Returning to \eqref{eq:sym} and observing that $(K^o)'$ and $(K^e)'$ are even and odd respectively, \begin{equation*}
			\begin{split}
				\int_{\bbS} 2 (\rd_\tht K*h^\varep_j) (\tht) h^\varep_j(\tht) d\tht  & = \iint_{\bbS\times\bbS} 2 (K^o + K^e)'(\tht-\tht') h^\varep_j (\tht) h^\varep_j (\tht') d\tht d\tht' \\
				& = \iint_{\bbS\times\bbS} 2 (K^o)'(\tht-\tht') h^\varep_j (\tht) h^\varep_j (\tht') d\tht d\tht' \\
				& = \int_{\bbS} 2 (\rd_\tht K^o *h^\varep_j) (\tht) h^\varep_j(\tht) d\tht  
			\end{split}
		\end{equation*} which can be estimated as \begin{equation*}
			\begin{split}
				& = \int_{\bbS} 2 (\rd_\tht K^o *h^\varep_j) (\tht^\varep_j(t)) h^\varep_j(\tht) d\tht + \int_{\bbS} 2( (\rd_\tht K^o *h^\varep_j) (\tht) - (\rd_\tht K^o *h^\varep_j) (\tht^\varep_j(t))) h^\varep_j(\tht) d\tht \\
				& = 2(\rd_\tht K^o *h^\varep_j) (\tht^\varep_j(t)) I^\varep_j(t) + O(\varep),
			\end{split}
		\end{equation*} where we have used that $h^\varep_j$ is $L^1$ uniformly in $\varep $ and $\rd_\tht K^o$ is Lipschitz. This gives that \begin{equation}\label{eq:ODE-I-varep}
			\begin{split}
				\left| \frac{d}{dt} I^\varep_{j} (t)  - 2 (\tilde{H}')^\varep_{j}(t, \tht^\varep_{j}(t)) I^\varep_{j} (t) \right| \le C\varep,
			\end{split}
		\end{equation} where by definition, \begin{equation*}
			\begin{split}
				(\tilde{H}')^\varep_{j} := (H')^\varep_{\ne j} + \rd_\tht K^o *h^\varep_j,
			\end{split}
		\end{equation*} which is Lipschitz continuous in $A_{j}$. Note that under the assumption (which we can bootstrap upon) of $h^\varep \to h$ in the sense of distributions, we have $( \tilde{H}')^\varep_{j}(t,\tht^\varep_{j}(t)) \to \partial_{\theta} H(t,\tht_{j}(t))$ where the latter is defined in \eqref{eq:Dirac-Hprime-explicit}. 
		
		Using this observation together with \eqref{eq:ODE-tht-varep}--\eqref{eq:ODE-I-varep} and noting that $I^{\varep}_{j}(t=0) = I_{j,0}$, $\tht^{\varep}_{j}(t=0) = \tht_{j,0}$ for all $\varep$ sufficiently small (possibly depending on $j$) gives that for any fixed $N\ge1$ and some $T>0$ small, we have convergence $\{ I^{\varep}_{j}(t), \tht^{\varep}_{j}(t) \} \to \{ I_{j}(t), \tht_{j}(t) \}$ for $t \in [0,T]$ and $j\le N$. This gives the desired convergence in the sense of distributions. 
	\end{proof}

	\subsection{Proof of Proposition \ref{prop:back-to-2D}}
	The goal is to show that $h\in C_*\left([0,T);\mathcal{M}\left(\mathbb{S}\right)\right)$ solution of \eqref{eq:vort-evo}--\eqref{eq:elliptic} defines a weak solution to the 2d Euler equations through \eqref{eq:vort-homog-spi}-\eqref{eq:vel-homog} in velocity forms. For $h\in C\left([0,T);L^1\left(\mathbb{S}\right)\right)$, we show moreover that $h$ defines a weak solution in vorticity form.
	\subsubsection{In vorticity form for $\boldsymbol{h_0\in L^1\left(\mathbb{S}\right)}$}
	To do so we write the 2d Euler equations \eqref{eq:2DEuler-vort} in polar coordinates 
	 \begin{equation}\label{eq:vorticity-strong}
	\partial_t \omega+u^r\partial_r \omega+\frac{1}{r}u^\theta \partial_\theta \omega=0.
	\end{equation} 
    To begin with, assume that $\omg(t,r,\theta)=h(t,\theta-\beta\ln r)$ with $h \in C^{1}(\bbS)$. Then $H, H' \in C^{1}(\bbS)$ as well, and in particular $\omg$ and $u$ are $C^{1}(\bbR^{2}\backslash\{0\})$, where we recall that \begin{equation}\label{eq:vel}
        \begin{split}
            u(t,r,\theta)=  -r H'(t,\theta -\beta \ln (r)) e^{r} + (
		2r H(t,\theta -\beta \ln (r))-r\beta H'(t,\theta -\beta \ln (r)) ) e^{\tht}.
        \end{split}
    \end{equation} Using the relation between $h$ and $H$, we see that a solution of \eqref{eq:vort-evo}--\eqref{eq:elliptic} with $h \in C^{1}(\bbS)$ solves \eqref{eq:vorticity-strong} pointwise, except at the origin. To deal with this problem we can just directly use the definition of a weak solution to \eqref{eq:vorticity-strong}: it should solve for any test function $\phi\in C^{\infty}_c\left(I\times \mathbb{R}^2\right)$ the identity 
	\[
	\int\limits_{(0,T)\times\mathbb{R}^2}\omega \partial_t\phi+\omega\underbrace{\left(\partial_r u^r+\frac{1}{r}\partial_\theta u^\theta\right)}_{=0}\phi+\omega u^r\partial_r \phi+\frac{1}{r}u^\theta\omega \partial_\theta \phi=-\int\limits_{\mathbb{R}^2}\omega_0\phi_0.
	\]
    To verify this, one simply rewrite the integral as the limit \[
	\lim_{\eps\to 0^+} \int\limits_{(0,T)\times(\mathbb{R}^2 \backslash B_{0}(\eps)) }\omega \partial_t\phi+\omega u^r\partial_r \phi+\frac{1}{r}u^\theta\omega \partial_\theta \phi\]
    which is possible thanks to uniform boundedness (in $\eps$) of the integrand in space and time. Then we simply integrate by parts in space: the boundary integral terms vanish as we take $\eps\to 0^+$ since $\rd_r(\omg u^r)$ and $r^{-1}\rd_\tht(\omg u^\tht)$ are again uniformly bounded in space and time. 

    In the case of $\omg(t,r,\theta)=h(t,\theta-\beta\ln r)$ with $h \in C\left([0,T], L^{1}(\bbS)\right)$, we note that $H \in  C\left([0,T], W^{2,1}(\bbS)\right)$ and  $H' \in  C\left([0,T], W^{1,1}(\bbS)\right)$. In particular, from \eqref{eq:vel} we see that $u\in  C\left([0,T], C^0(\mathbb{R}^2)\right)$. Now, to verify that it is a weak solution in $\bbR^{2}$, one can mollify $h, H, H'$ in $\theta$ to reduce to the case of $h \in C^1(\bbS)$, while the error terms coming from mollification are small thanks to the $L^1$ bound of $h$.

\subsubsection{In velocity form for $\boldsymbol{h_0\in \mathcal{M}\left(\mathbb{S}\right)}$}

When $h$ is merely a measure, we need to use the velocity formulation to prove that the associated velocity defines a weak solution in $\bbR^2$. We write the 2d Euler equations on the velocity in polar coordinates 
	\[
	\begin{cases}
		\partial_t u^r+u^r\partial_ru^r+\frac{1}{r}u^\theta \partial_\theta u^r=-\partial_r p\\
		\partial_t u^\theta+u^r\partial_ru^\theta+\frac{1}{r}u^\theta \partial_\theta u^\theta=-\frac{1}{r}\partial_\theta p. 
	\end{cases}
	\]
	Thus a weak solution of the 2d Euler equation on an interval $[0,T)\subset \mathbb{R}$ solves for $(\phi^r,\phi^\theta)\in C^{\infty}_c\left(I\times \mathbb{R}^2\right)$
	\begin{equation}\label{eq: weak euler}
	\begin{cases}
		\int\limits_{(0,T)\times\mathbb{R}^2}u^r \partial_t\phi^r+(u^r)^2 \partial_r \phi^r+\frac{1}{r}u^\theta u^r \partial_\theta \phi^r=-\int\limits_{\mathbb{R}^2}u^r_0\phi^r_0-\int\limits_{(0,T)\times\mathbb{R}^2} p \partial_r \phi^r,\\
		\int\limits_{(0,T)\times\mathbb{R}^2}u^\theta \partial_t\phi^\theta+u^r u^\theta \partial_r \phi^\theta+\frac{1}{r}(u^\theta)^2 \partial_\theta \phi^\theta=-\int\limits_{\mathbb{R}^2}u^\theta_0\phi^\theta_0-\int\limits_{(0,T)\times\mathbb{R}^2} p \frac{1}{r}\partial_\theta \phi^\theta,
	\end{cases}
	\end{equation}
	where we again used the incompressibility condition $\partial_r u^r+\frac{1}{r}\partial_\theta u^\theta=0$.

 To proceed, we need to compute the form of the pressure under logarithmic spiral symmetry. The ansatz for the pressure is $p= r^2 P(t,\tht-\bt\ln r)$ and we  compute the gradient in polar \begin{equation*}
	    \begin{split}
	        \nb p = (2rP-\beta r P') e^{r} +  r P'  e^{\tht}
	    \end{split}
	\end{equation*} which gives 
 \[ -2rP=\partial_t \left(u^r+\beta u^\theta\right)+u^r\partial_r\left(u^r+\beta u^\theta\right)+\frac{1}{r}u^\theta \partial_\theta \left(u^r+\beta u^\theta\right).
 \]
 Now $u^r+\beta u^\theta=r\left(2\beta H-(1+\beta^2) H' \right)$ thus from \eqref{eq:H}--\eqref{eq:H-prime} we get
 \begin{align*}2P&=-4\beta H H' -2\beta\bfK(8\bt - 3(1+\bt^2)\rd_\tht)\left[ (H' )^2  \right]+(1+\beta^2)\left( 2H H''- (H' )^2 + 4 \bfK (3-\bt\rd_\tht)\left[ (H' )^2 \right]\right)\\
&-H'  \left(2\beta H+(1+\beta^2)\beta H'' -(1+3\beta^2)H' \right)+\left(2H-\beta H' \right)\left(2\beta  H' -\left(1+\beta^2\right) H''\right)\end{align*}
in particular we observe that the worst term $(1+\beta^2)\beta  H' H''$, which is not well defined in the case of $h\in \mathcal{D}\left(\mathbb{S}\right)$ cancels out, and we get
\[2P=-2\beta\bfK(8\bt - 3(1+\bt^2)\rd_\tht)\left[ ( H' )^2  \right]+(1+\beta^2)\left( 4 \bfK (3-\bt\rd_\tht)\left[ ( H' )^2 \right]\right)
     - H'  \left(2\beta H-\beta^2 H' \right).\] Recall that $h \in \calD(\bbS)$ gives $H \in W^{1,\infty}(\bbS)$ and $H' \in L^\infty(\bbS)$. Using this (with regularizing property of $\bbK \rd_\tht$), we have in particular for $h\in C_*\left([0,T);\mathcal{D}\left(\mathbb{S}\right)\right)$ that $P\in C_*\left([0,T);L^\infty\left(\mathbb{S}\right)\right)$.

\medskip 

     \noindent We now proceed to the proof that $h$ provides a weak solution to \eqref{eq: weak euler}. To do so, we  consider $h^\epsilon$ as the mollification of our solution given in Section \ref{sec:cvg moll}. Let $u^\epsilon$ and $p^\epsilon$ be the respective velocity and pressure associated to $h^\epsilon$. As they are $W^{1,\infty}$ solutions (and actually smooth outside of $0$) to the 2d Euler equation they solve
     \[
    \begin{cases}		\int\limits_{(0,T)\times\mathbb{R}^2}\left(u^r \right)^\epsilon\partial_t\phi^r+\left((u^r)^\epsilon\right)^2 \partial_r \phi^r+\frac{1}{r}(u^\theta)^{\epsilon} (u^r)^{\epsilon} \partial_\theta \phi^r=-\int\limits_{\mathbb{R}^2}(u^r_0)^{\epsilon}\phi^r_0-\int\limits_{(0,T)\times\mathbb{R}^2} p^\epsilon \partial_r \phi^r,\\
		\int\limits_{(0,T)\times\mathbb{R}^2}(u^\theta)^{\epsilon} \partial_t\phi^\theta+(u^r)^{\epsilon} (u^\theta)^{\epsilon} \partial_r \phi^\theta+\frac{1}{r}\left((u^\theta)^{\epsilon}\right)^2 \partial_\theta \phi^\theta=-\int\limits_{\mathbb{R}^2}(u^\theta_0)^\epsilon \phi^\theta_0-\int\limits_{(0,T)\times\mathbb{R}^2} p^\epsilon \frac{1}{r}\partial_\theta \phi^\theta.
	\end{cases}
	\] This can be proved along the lines of the proof that $h \in C^1(\bbS)$ gives rise to weak solutions to the vorticity equation. 
 Now, by the results of Section \ref{sec:cvg moll} $h^\epsilon$ converges to $h$ in the sense of measures. This shows that $(H'')^\epsilon$ again converges to $H''$ in the sense of measures (the same holds for $H$ and $H'$ for a stronger reason), giving that $u^\epsilon \to u$ and $p^\epsilon \to p$ in the sense of measures, where $u$ and $p$ are the velocity and pressure generated from $h$. In particular, $u$ and $p$ verify \eqref{eq: weak euler} by passing to the limit $\eps\to 0$ in the previous identities. This finishes the proof. \qedsymbol

	\section{Long time dynamics and singularity formation}\label{sec:longtime}

	\subsection{Convergence for bounded solutions}

 \begin{proof}[Proof of Theorem \ref{thm:conv}]
	Recall that we are assuming $\bt>0$. For $h_0 \in L^\infty$, we have that for all $t>0$, the solution $h(t,\cdot)$ satisfies \begin{equation*}
		\begin{split}
			-\nrm{h_0}_{L^\infty} \le h(t,\tht) \le \nrm{h_0}_{L^\infty}. 
		\end{split}
	\end{equation*} In particular, \begin{equation*}
		\begin{split}
			I(t) := \int h(t,\tht) d\tht \ge -2\pi\nrm{h_0}_{L^\infty}  
		\end{split}
	\end{equation*} and since the left hand side is strictly decreasing in time (unless $h_0$ is a constant), \begin{equation*}
		\begin{split}
			I(t) \longrightarrow \bfI_{+}
		\end{split}
	\end{equation*} for some constant $\bfI_{+} < I(0)$ as $t\to\infty$. Then, integrating \eqref{eq:mon} in time, \begin{equation*}
		\begin{split}
			8\bt\int_0^\infty \int (H')^2 d\tht dt = I(0) - \bfI_{+},
		\end{split}
	\end{equation*} which gives $H \in L^2([0,\infty); \dot{H}^1(\bbS)) =: L^2_t \dot{H}^1_\tht$.  Next, from the equation \eqref{eq:H} for $H$  \begin{equation*}
		\begin{split}
			\nrm{\rd_tH}_{L^2} &\le  2\nrm{H H'}_{L^2} + \nrm{\bfK(8\bt - 3(1+\bt^2)\rd_\tht)\left[ (H')^2  \right]}_{L^2} \\
			&\le C \nrm{H}_{W^{1,\infty}} \nrm{H'}_{L^2} \le C \nrm{h}_{L^\infty} \nrm{H'}_{L^2} \le C \nrm{h_0}_{L^{\infty}} \nrm{H'}_{L^2}. 
		\end{split}
	\end{equation*} This gives $\rd_tH \in L^2_tH^1_\tht$ as well. We have that $\int H d\tht = \frac14 \int h d\tht \to \frac14 \bfI_{+}$. Applying Aubin--Lions lemma to the sequence of functions $\left\{ H(\cdot+t_n, \cdot) \right\}$ defined on $[0,\infty)\times\bbS$ (here, $t_n\ge0 $ is an arbitrary increasing sequence), we obtain a convergent subsequence in $L^2_tH^1_\tht$. The limit must be equal to the constant $\frac14 \bfI_{+}$, and therefore is independent of the choice of a subsequence; $H(t,\cdot)\to \frac14 \bfI_{+}$ in $L^2(\bbS)$. Since $h\in L^\infty(\bbS)$ uniformly in time, the convergence holds in $H^{-a}(\bbS)$ for any $a>0$ in terms of $h$. \end{proof}

	\subsection{Trichotomy for $L^p$ data}
\begin{proof}[Proof of Theorem \ref{thm:tri}]
The trichotomy of behavior follows from the analysis of $I(t)$ when $h_0\in L^p$. We work again with $\beta>0$ and $h_0$ not identically constant, then $I(t)$ is a strictly decreasing function of time and one the following three scenarios must occur. There exists $T^*\in (0,+\infty]$ such that either
	\begin{itemize}
		\item $T_*=+\infty$ and there exists $\bfI_{+}\in \mathbb{R}$ such that $I(t)\underset{t\to +\infty}{\longrightarrow}\bfI_{+}$,
		\item $T_*<+\infty$,
		\item $T_*=+\infty$ and $I(t)\underset{t\to +\infty}{\longrightarrow}-\infty$.
	\end{itemize}
	The only point requiring more analysis is the first one. Indeed as for the case $h_0\in L^\infty$ we get that $H\in L^2_t\dot{H}^1_\theta$. Now $I(t)$ being bounded combined with $H\in L^2_t\dot{H}^1_\theta$ implies that $h\in L^\infty_t L^1_\theta$. Again from \eqref{eq:H}
	\[
	\left\Vert \partial_t H \right\Vert_{L^2} \leq C\left\Vert H\right\Vert_{W^{1,\infty}}\left\Vert H'\right\Vert_{L^2}\leq C \left\Vert h\right\Vert_{L^\infty_{(0,+\infty)}L^1_\theta}\left\Vert H'\right\Vert_{L^2},
	\]
	and the proof follows as in the previous paragraph.    
\end{proof}

	\subsection{Singularity formation for Diracs}

    \begin{proof}[Proof of Theorem \ref{thm:blowup-Dirac}]
        
	We assume $\bt>0$. To begin with, we note that for each fixed $N$, if $h = \sum_{i=1}^{N} I_i\dlt(\tht-\tht_i)$ then $|\sum_{i=1}^{N} I_i|$ is controlled $\nrm{H'}_{L^{2}}$.\footnote{On the other hand, it is not correct for the norm $\sum_{i=1}^{N} |I_i|$.} This is clear in the case $N=1$. To see this for $N=2$, note that $\nrm{H'}_{L^2}$ is a bilinear form of $I_1$ and $I_2$ with coefficients depending only on $\tht_1-\tht_2$. Then, it suffices to observe that this bilinear form is strictly positive when $I_1+I_2 \ne 0$ and converges to a strictly positive quadratic form of $(I_1+I_2)$ as $\tht_1-\tht_2 \to 0$. The general case follows from a continuity argument and induction in $N$. That is, \begin{equation*}
		\begin{split}
			\nrm{H'}_{L^2} \ge c_N|\sum_{i=1}^{N} I_i|
		\end{split}
	\end{equation*} for some $c_N>0$ depending on $N$. Applying \eqref{eq:mon-di}, we see that \begin{equation}\label{eq:Ric}
		\begin{split}
			\frac{d}{dt} ( \sum_{i=1}^{N} I_i ) \le -8\bt c_{N}^{2} ( \sum_{i=1}^{N} I_i )^{2}. 
		\end{split}
	\end{equation} In particular, if it happens that $( \sum_{i=1}^{N} I_i ) < 0$ for some $t_0$, then necessarily $( \sum_{i=1}^{N} I_i(t) ) \to -\infty$ as $t$ approaches some $T>t_0$. Therefore, for finite time singularity to not occur, it is \textit{necessary} that $\sum_{i=1}^{N} I_i(t) \ge 0$ for all $t\ge0$. Furthermore, we can exclude the case $\sum_{i=1}^{N} I_i(t) = 0$ (for any $t\ge0$) since $\sum_{i=1}^{N} I_i(t)$ is \textit{strictly} decreasing (see \eqref{eq:mon-di}), unless $h$ is trivial. Then, we have that  $\sum_{i=1}^{N} I_i(t) > 0$ for all $t\ge0$ and \eqref{eq:Ric} tells us that  $\sum_{i=1}^{N} I_i(t) \to 0$ as $t\to\infty$.\end{proof}
    
	\subsection{Case study of $m$ symmetric Dirac deltas}\label{subsec:case-study}
	
	In this section, we revisit the case of $m$-fold symmetric Dirac deltas where $m\ge1$ is an integer. Namely, we consider the dynamics of the solution of the form \begin{equation}\label{eq:initial-symm}
		\begin{split}
			h(t,\cdot) = I_{0}(t) \sum_{j=0}^{m-1} \dlt_{\tht_{j}(t)}
		\end{split}
	\end{equation} where $\tht_{j}(t) = \tht_{0}(t) + 2\pi j/m$ for $j = 1, \cdots, m-1$. The $m$-fold symmetry is preserved in time, and the solution is characterized by $(I_{0}, \tht_{0})$. It is then natural to introduce the $m$-fold symmetric kernel $K^{m}$ by \begin{equation*}
		\begin{split}
			K^{m}(\tht) = \sum_{j=0}^{m-1} K(\tht+ 2\pi j/m ). 
		\end{split}
	\end{equation*} We give an explicit form of this symmetrized kernel in the Appendix. Furthermore, we can simply take the spatial domain to be $\bbS^{m}$ which is $(0,2\pi/m)$ with endpoints identified with each other. The system of equations for $(I_{0}, \tht_{0})$ reads \begin{equation}\label{eq:dirac-m-fold-tht}
		\begin{split}
			\frac{d}{dt} \tht_{0}(t) = 2H(t,\tht_{0}(t)) = 2 K^{m}(0)I_{0}(t)
		\end{split}
	\end{equation} \begin{equation}\label{eq:dirac-m-fold-I}
		\begin{split}
			\frac{d}{dt} I_{0}(t) = 2 (K^{m})'(0) (I_{0}(t))^2. 
		\end{split}
	\end{equation} We see that the equation for $I_{0}$ does not involve the other variable $\tht_{0}$, and the solution is simply \begin{equation*}
		\begin{split}
			I_{0}(t) = \frac{I_{0}(0)}{1 - 2 (K^{m})'(0) I_{0}(0) t} . 
		\end{split}
	\end{equation*} Depending on the sign of $2 (K^{m})'(0) I_{0}(0)$, we have either finite-time blow up or decay of rate $1/t$ as $t$ becomes large. The constant $ (K^{m})'(0) $ can be explicitly determined as a function of $m,\bt$ and is given in Remark \ref{rem:derv ker 0}. Assume that $ (K^{m})'(0) I_{0}(0) > 0$, so that the Dirac solution blows up at some $T^*>0$. It is an interesting exercise to see what happens to the sequence of patch regularizations of \eqref{eq:initial-symm}: \begin{equation}\label{eq:initial-patch}
		\begin{split}
			h_{0}^{\varep} := \frac{I_{0}(0)}{2\varep} \sum_{j=0}^{m-1} \mathbf{1}_{[\tht_{j}(0)-\varep, \tht_{j}(0)+\varep]}. 
		\end{split}
	\end{equation} While there is a global solution associated with $h_{0}^{\varep} $, one can show that as $t\to T^*$, the support of $h^{\varep}(t) $ occupies almost all of the spatial domain, so that in particular $\nrm{ h^{\varep} (T^*) }_{L^{1}} = O(\varep^{-1})$, which blows up as $\varep\to 0^+$. 
	
	\subsection{Case study of two non-symmetric Diracs}\label{subsec:case-study-2}

    In this section, we study the evolution of two Dirac deltas in the case $m=1$. For simplicity, we shall assume that they evolve in a self-similar fashion: their distance in $\bbS$ is fixed while the amplitudes are proportional to $1/t$. To this end, we recall the system \eqref{eq:ODE} for two Diracs: \begin{equation*}
	\begin{split}
		I_1' &= 2K'(0) I_1^2 + 2K'(\tht_1-\tht_2) I_1 I_2, \\
		I_2' &= 2K'(0) I_2^2 + 2K'(\tht_2-\tht_1) I_1 I_2, \\
		\tht_1 ' &= 2I_1 K(0) + 2I_2 K(\tht_1-\tht_2), \\
		\tht_2' &= 2I_2 K(0) + 2I_1 K(\tht_2-\tht_1), 
	\end{split}
\end{equation*} and assume a solution of the form \begin{equation*}
     I_1(t)=A_1 t^{-1}, \quad I_2(t)=A_2t^{-1},\quad \tht_1(t)-\tht_2(t) = d
\end{equation*} where $A_1,A_2,d$ are constants. We may assume further that $0<d\le \pi$. Under these assumptions, the ODE reduces to the following system of algebraic equations \begin{equation}\label{eq:two-dirac-sys}
	\begin{split}
		-1 &= 2A_1 K'(0) + 2A_{2} K'(d), \\
            -1 &= 2A_{2} K'(0) + 2A_{1} K'(-d),\\
            0 & = (A_{1}-A_{2})K(0) + A_{2} K(d) - A_{1}K(-d). 
	\end{split}
\end{equation} Assuming that $(K'(0))^{2} - K'(d)K'(-d) \ne 0$, $A_{1}$ and $A_{2}$ are uniquely determined from the first two equations in terms of $d$, and we are left with the single equation \begin{equation}\label{eq:two-dirac}
    K(0)(  K'(-d)-K'(d)) + K(d)(K'(0)-K'(-d)) + K(-d)(K'(d)-K'(0)) = 0. 
\end{equation} Even when $(K'(0))^{2} - K'(d)K'(-d) = 0$, a solution of \eqref{eq:two-dirac} gives (infinitely many) solutions to \eqref{eq:two-dirac-sys}. We clearly see that $d=\pi$ solves \eqref{eq:two-dirac}, which simply corresponds to the symmetric self-similar blow up of two Dirac deltas. We now consider the function 
\[
F(\beta,d)=K(0)(  K'(-d)-K'(d)) + K(d)(K'(0)-K'(-d)) + K(-d)(K'(d)-K'(0)),
\]
then computing 
\[
K_{\beta}(\theta)= -\frac{\sin(2\theta)}{8\pi\beta}+O(\beta) \text{ and } K_{\beta}'(\theta)= -\frac{\cos(2\theta)}{4\pi\beta}+O(\beta)
\]
we get
\[
F(\beta,d)\underset{\beta \to 0}{\sim} \frac{1}{16\pi^2 \beta^2}\sin(2d)(1-\cos(2d)).
\]
Thus considering the function $\tilde{F}=16\pi^2 \beta^2 F$ we observe that $\tilde{F}$ is $C^1$ in neighborhood of $(0,\frac{\pi}{2})$, $\tilde{F}(0,\frac{\pi}{2})=0$ and $\partial_d\tilde{F}(0,\frac{\pi}{2})=-4$ thus an application of the implicit function theorem immediately yields the existence and uniqueness of a continuum of (non symmetric) solutions in the form $(\beta,d(\beta))$ in a neighborhood of $(0,\frac{\pi}{2})$. This exactly the content of \cite[Theorem 1]{CKW22} in the case $M=2$.

\begin{remark}\label{Rem: contains ckw}
    The case of $2$ Diracs is not special indeed an analogous reasoning for $M$ Diracs yields a system of $1+M-1$ equations analogous to \eqref{eq:two-dirac}. It can be then observed that the problem reduces to finding the zeros of another function $F$ of $1+M-1$ variables, $\beta$ and the $M-1$ differences of angles, into $\mathbb{R}^{M-1}$ for which the implicit function theorem can applied near the point $\beta=0$ and $(\frac{\pi}{M},\cdots,\frac{(M-1)}{M}\pi)$ if the differential is not singular in the $M-1$ variables. This is the content of Theorem 1 of \cite{CKW22} where the cases $M\in \{2,3,5,7,9\}$ are covered.
\end{remark} 

\noindent We get back to the case $M=2$ and study the behavior of $F$ in the limit $\beta$ goes to infinity. In this case, it is not very difficult to verify that \begin{equation*}
    K_{\bt}(\tht) = \frac{1}{8\pi} + \frac{(2\pi-\tht)\tht}{4\pi \bt^2} + O(\bt^{-3}),\qquad K_{\bt}'(\tht) = \frac{\pi-\tht}{2\pi\bt^2} + O(\bt^{-3}).
\end{equation*} This gives that, as $\bt\to\infty$, \begin{equation*}
    F(\bt,d)\sim \frac{1}{4\pi^2 \bt^4} (2\pi-d)(\pi-d)d + O(\bt^{-6}).  
\end{equation*} This limit can be shown to be uniform in $C^1$ on $0\le d \le \pi$. From this we deduce that for $0<d<\pi$, there are no zeros of $F(\bt,d)$ as long as $\bt$ is taken to be sufficiently large. Combining this computation with the previous one, we can arrive at the following bifurcation result in the case of two Dirac deltas. 
\begin{proposition}
    There exist some $\bt_{0},\bt_{1}>0$ such that the following holds: \begin{itemize}
        \item for all $0<\bt<\bt_{0}$, there is only one zero of $F(\bt,d)=0$ in $0<d<\pi$. This unique zero converges to $\pi/2$ as $\bt\to0$, and 
        \item for all $\bt>\bt_{1}$, there are no zeroes of $F(\bt,d)=0$ in $0<d<\pi$.
    \end{itemize} In particular, there exists at least one bifurcation point of non-symmetric solutions from the symmetric one.
\end{proposition}
\subsection{Some open questions}
The following observations would be interesting to investigate further potentially shedding some light on some new phenomena emerging in the long time dynamics for solutions of the 2d Euler equations.
\begin{enumerate}
    \item Based on numerical simulations it seems like $\beta_0=\beta_1=\beta_b$ and the two Diracs system exhibits a unique bifurcation.
    \item For $\beta> \beta_b$, the symmetric configuration seems to be stable and is the unique global attractor of the blow up dynamics exhibited in Theorem \ref{thm:blowup-Dirac}.
    \item For $\beta< \beta_b$, bifurcation occurs, the new configuration becomes the stable and unique global attractor of the blow up and the symmetric configuration becomes unstable.
    \item It would be interesting to investigate further this bifurcation phenomena for three or more Diracs. 
    \item One can investigate the role of $m$-fold symmetry in the bifurcation of non-symmetric solutions. 
\end{enumerate}
	
	\appendix

	\section*{Appendix}
	\section{Derivation of the kernel}

        In this section, we give an explicit form of the kernel. 
 
	\begin{lemma}\label{lem:ker comp}
		Let $h, H$ solve the elliptic equation \eqref{eq:elliptic} for some $\bt\ne0$ and suppose moreover they are $m$-fold symmetric in $\theta$ for $m\geq 1$. Then for $\theta \in [0, \frac{2\pi}{m})$ we have
		\[
		H(\tht)=\int_{0}^{2\pi} h(\tht') K^m_{\beta}(\tht-\tht') d\tht',
		\]
		with 
  \[
 K^m_{\beta}(\theta)= \frac{1}{4}\Re \left[\frac{e^{\frac{2(\beta-i)}{m(1+\beta^2)}(m\theta-\pi)}}{\sin\left(\frac{2\pi(1+i\beta)}{m(1+\beta^2)}\right)}\right].
  \]
	\end{lemma}
	\begin{remark}\label{rem:derv ker 0}
We record the values
\[
K_\beta^m(0)=\frac{1}{4}\Re\left[\cot\left(\frac{2\pi(1+i\beta)}{m(1+\beta^2)}\right)\right],
\]
and
\[
{K_\beta^m} '(0)=\frac{\Re\left[(\beta-i)\cot\left(\frac{2\pi(1+i\beta)}{m(1+\beta^2)}\right)\right]}{2(1+\beta^2)}=\frac{4\beta K_\beta^m(0)+\Im\left[\cot\left(\frac{2\pi(1+i\beta)}{m(1+\beta^2)}\right)\right]}{2(1+\beta^2)}.
\]
	\end{remark}
	\begin{proof}
 We give the proof in the case $m = 1$, as the general case can be obtained similarly. 
		Using the Fourier series \begin{equation*}
			\begin{split}
				h(\tht) = \sum_{n\in\bbZ} \hat{h}_n e^{in\tht},\qquad H(\tht) = \sum_{n\in\bbZ} \hat{H}_n e^{in\tht}, 
			\end{split}
		\end{equation*} we obtain the relation \begin{equation}\label{eq:elliptic-Fourier}
			\begin{split}
				\hat{H}_n = \frac{1}{4-4\bt in - (1+\beta^2)n^2} \hat{h}_n= \left[  \frac{1}{ \frac{2i}{i-\beta}- n } + \frac{1}{\frac{2i}{i+\beta} + n}  \right] \frac{\hat{h}_n}{4}.
			\end{split}
		\end{equation} 
			Thus
			\begin{align}\label{eq:new K 1}
				K_{\beta}(\theta)&=\frac{1}{8\pi}  \sum_{n\in\bbZ}  \left[  \frac{1}{ \frac{2i}{i-\beta}- n } + \frac{1}{\frac{2i}{i+\beta} + n}  \right] e^{in\tht }=\frac{1}{8\pi}  \sum_{n\in\bbZ}  \left[  \frac{1}{- \frac{2i(i+\beta)}{1+\beta^2}- n } + \frac{1}{\frac{2i(-i+\beta)}{1+\beta^2} + n}  \right] e^{in\tht }\nonumber\\
				&=\frac{1}{8\pi}  \sum_{n\in\bbZ} \left[\frac{1}{\frac{2(1+i\beta)}{1+\beta^2} + n} - \frac{1}{ \frac{2(-1+i\beta)}{1+\beta^2}+ n }   \right] e^{in\tht }.
			\end{align}
			We use the following standard Fourier series computation found for example in \cite{stein2003princeton}.
			\begin{lemma}\label{lem:FS}
				For $\alpha \in \mathbb{C}\setminus \mathbb{N}$ and $\theta \in \mathbb{R}$
				\[
				\sum_{n\in \mathbb{N}}\frac{e^{in\theta}}{n+\alpha}=\frac{\pi}{\sin(\pi\alpha)}e^{i(\pi-\theta)\alpha}.
				\]
			\end{lemma}
			\noindent Applying the previous lemma to \eqref{eq:new K 1} we get for $\beta\neq 0$
			\begin{equation}\label{eq:cln K}
				K_{\beta}(\theta)=\frac{1}{8}\left[\frac{e^{i\frac{2(1+i\beta)}{1+\beta^2}(\pi-\theta)}}{\sin\left(\frac{2\pi(1+i\beta)}{1+\beta^2}\right)}-\frac{e^{i\frac{2(-1+i\beta)}{1+\beta^2}(\pi-\theta)}}{\sin\left(\frac{2\pi(-1+i\beta)}{1+\beta^2}\right)}\right]=\frac{1}{4}\Re \left[\frac{e^{\frac{2(\beta-i)}{1+\beta^2}(\theta-\pi)}}{\sin\left(\frac{2\pi(1+i\beta)}{1+\beta^2}\right)}\right].
				\end{equation}
  
			\noindent Noting that the $m$-fold symmetric Kernel $m\geq 1$ is given by 
			\begin{align*}
				K_\beta^m(\theta)&=\frac{1}{8\pi}  \sum_{n\in\bbZ}  \left[  \frac{1}{ \frac{2i}{i-\beta}- mn } + \frac{1}{\frac{2i}{i+\beta} + mn}  \right] e^{imn\tht }=\frac{1}{8\pi}  \sum_{n\in\bbZ} \left[\frac{1}{\frac{2(1+i\beta)}{1+\beta^2} + mn} - \frac{1}{ \frac{2(-1+i\beta)}{1+\beta^2}+ mn }   \right] e^{imn\tht }\\
				&=\frac{1}{8\pi m}  \sum_{n\in\bbZ} \left[\frac{1}{\frac{2(1+i\beta)}{m(1+\beta^2)} + n} - \frac{1}{ \frac{2(-1+i\beta)}{m(1+\beta^2)}+ n }   \right] e^{in m\tht },
			\end{align*}
			Thus making the change of variables $\theta\leftrightarrow m\theta$ and $(1+\beta^2)\leftrightarrow m(1+\beta^2)$ we get for $\theta\in (0,\frac{2\pi}{m})$ the desired result.
		\end{proof}

	\begin{figure}
		\centering
		\includegraphics[scale=0.6]{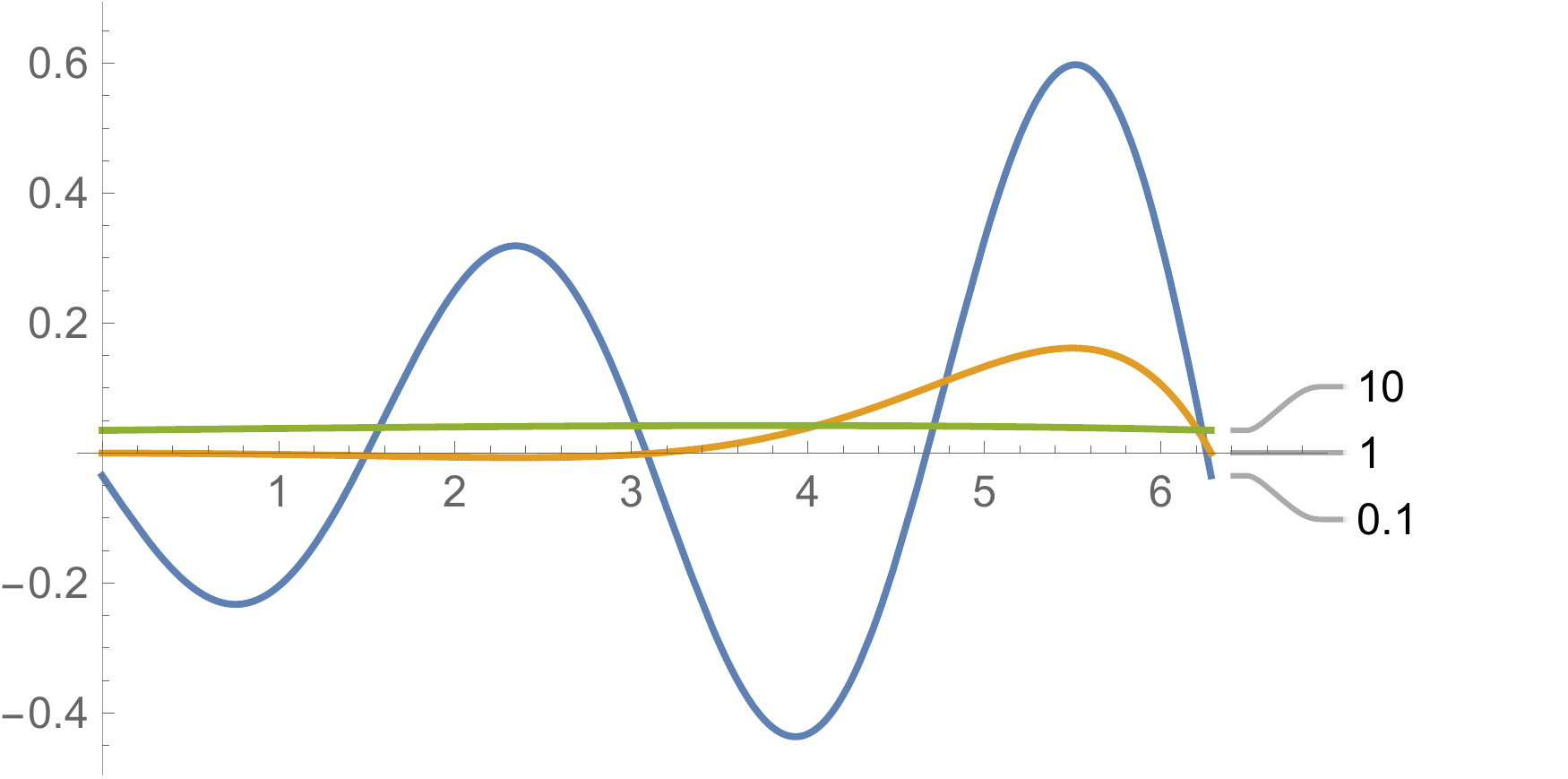}   
		\caption{Plots of $K^1_\bt(\tht)$ for some values of $\bt$} \label{fig:kernel}
	\end{figure}

\begin{example}
			Let us take $\beta = 1$ and $m=1$. Then, on the interval $(0,2\pi)$, we explicitly have \begin{equation*}
				\begin{split}
					K_{1}^{1} (\tht) = \frac{1}{2}\frac{\sin(\theta)e^\theta}{1-e^{2\pi}}, \quad (K^{1}_{1})'(\tht) = \frac{1}{2}\frac{(\sin(\theta)+\cos(\theta))e^\theta}{1-e^{2\pi}}.
				\end{split}
			\end{equation*} The plot is given in Figure \ref{fig:kernel}. 
		\end{example}

   \section{Self-similar logarithmic vortex sheets}\label{sec:self-similar}

   In this section, we show how our formulation of logarithmic vortex sheets corresponds to the traditional one which goes back to the work of Prandtl. We follow the notation of \cite{CKW21,CKW22}: in the case of one branch, they write 
   \begin{equation*}
				\begin{split}
					Z(t,\tht) & = t^{\mu} \exp( i\tht + (\tht-\tht_{0})/\beta), \\
     \Gmm(t,\tht) & = g t^{2\mu - 1} \exp( 2(\tht-\tht_{0})/\beta)
				\end{split}
			\end{equation*} where $Z$ and $\Gmm$ correspond to the location and circulation of the spiral in $\bbR^2$ parameterized by $\tht\in \bbR$, respectively. Here $\tht_{0},g,\mu$ are constants. The vorticity is then given by the sheet supported on the set $\Sigma(t) = \left\{  Z(t,\tht) : \tht \in \bbR \right\} $, characterized by local circulation \begin{equation*}
			    R^{-2}\iint_{|x|\le R} \omg(t,x) dx = gt^{-1}.
			\end{equation*} 
   Taking $t = 1$ to be the initial time for simplicity, the above ansatz corresponds to taking our $h$ function to be \begin{equation*}
       h(t=1,\tht) = 2g \dlt(\tht-\tht_{0}). 
   \end{equation*} (Here, the factor 2 comes from comparing the circulation formula with \eqref{eq:entropy spirals}.) Using our formulation, for the Dirac delta solution corresponding to $h(t=1,\cdot)$ to be self-similar, we need to have \begin{equation*}
       h(t,\tht) = \frac{2g}{t} \dlt(\tht - \Theta(t,\tht_{0})), \quad \Theta(t,\tht_0) = \tht_0 - \bt\mu \ln(t)
   \end{equation*}
   and then we obtain two consistency equations by comparing these with the ODE system \eqref{eq:ODE} at $t=1$, which are nothing but exactly the ones given in \cite[Corollary 1.3]{CKW21}. Indeed we are looking for $I(t)=\frac{2g}{t}$ and $\Theta(t,\theta_0)$ to be solutions of \eqref{eq:ODE} which imposes in the case of one branch:
   \[
   I'(t)=2I(t)^2K'(0) \text{ and } \Theta'(t,\theta_0)=2I(t)K(0),
   \]
   where $K(0)$ and $K'(0)$ are given explicitly in Remark \eqref{rem:derv ker 0} thus 
   \[
   -\frac{2g}{t^2}=\frac{8g^2 K'(0)}{t^2} \text{ and } -\frac{\beta \mu}{t}=\frac{4g}{t}K(0),
   \]
   which finally gives
\[
-4gK'(0)=1 \text{ and }gK(0)=-\beta\mu.
\]
   One may similarly consider the case of $m$ branch spiral sheets as well. 
	
		\bibliographystyle{plain}
		\bibliography{log-spiral}
		
	\end{document}